\documentclass[reqno,11pt]{amsart}
\usepackage{amsmath,amsfonts,amssymb,amsxtra,latexsym,amscd,enumerate,enumitem,amsthm,verbatim,here}
\usepackage{mathtools}
\usepackage{array}
\usepackage{booktabs}
\usepackage[table,dvipsnames]{xcolor}
\usepackage{cancel}

\usepackage[margin=1.2in]{geometry}
\numberwithin{equation}{section}

\usepackage{hyperref}
\definecolor{myrefred}{RGB}{180,0,0}
\definecolor{myciteblue}{RGB}{0,70,140} 
\hypersetup{
    colorlinks=true,
    pdfstartview=FitV,
    linkcolor=myrefred,      
    citecolor=myciteblue,    
    urlcolor=black
}
\definecolor{labelkey}{rgb}{0.6,0,0}

\newcommand{\pr}{\partial}

\newcommand{\R}{\mathbb{R}}
\newcommand{\N}{\mathbb{N}}

\newcommand{\Z}{\mathbb{Z}}
\newcommand{\la}{\langle}
\newcommand{\ra}{\rangle}

\newcommand{\lo}{{\mathrm{low}} }
\newcommand{\hi}{{\mathrm{high}} }

\newcommand{\hs}{{\mathrm{H}}}
\newcommand{\ks}{{\mathrm{K}}}
\newcommand{\sgm}{\sum_{b=0}^{n-1}t^{1-(b+1)\alpha}\Gamma_b}

\def\K{\mathcal{K}}

\def\H{\mathcal{H}}
\def\pr{{\partial}}

\def\nab{\nabla}
\def\les{\lesssim}

\numberwithin{equation}{section} 

\newtheorem{theorem}{Theorem}

\newtheorem{lemma}{Lemma}[section]

\newtheorem{proposition}[lemma]{Proposition}

\theoremstyle{remark}
\newtheorem{remark}[lemma]{Remark}

\begin{document}
\title[A common structure for modified scattering along  dispersive rays]{A common structure for modified scattering: Vlasov--Riesz, Hartree, and their coupling}
\author{Wenrui Huang}
\address{Brown University}
\email{wenrui\_huang@brown.edu}
\author{Mengyi Xie}
\address{Yale University}
\email{mengyi.xie@yale.edu}

\begin{abstract}
We identify a common action--density structure for the Hartree and Vlasov--Riesz equations and their coupling. Along free rays, the feedback between nonlinear actions and rescaled densities governs wave phase modulations and kinetic momentum translations. This structure persists under the exchange of sources, providing convergent modified profiles and an explicit recursive construction of the asymptotic corrections. We thereby establish small-data global existence and modified scattering in three spatial dimensions for the Hartree equation, the Vlasov--Riesz equation, and the coupled Vlasov--Hartree system, with interaction potential $\lambda|x|^{-\alpha}$, $\lambda\in\mathbb R$ and $0<\alpha\leq 1/2$. In this strongly long-range regime, the asymptotic action contains finitely many successive corrections beyond its leading term. To our knowledge, these are the first forward modified-scattering results for all three models throughout this range.
\end{abstract}

\maketitle

\setcounter{tocdepth}{1}
\tableofcontents


\section{Introduction}\label{sec:intro}
We study the long-time dynamics of three mean-field models in three spatial dimensions with interaction potential
\begin{equation}\label{eq:Riesz_potential}
    V_\alpha(x):=\lambda|x|^{-\alpha}, \qquad 0<\alpha\leq \frac{1}{2},\qquad \lambda\in \R. 
\end{equation}
The Vlasov--Riesz equation is
\begin{equation}\label{VR}
    \begin{split}
        \partial_t f+v\cdot \nabla_x f+F\cdot \nabla_vf =0,\qquad F=-\nabla_x \Psi, \\
        \Psi=V_\alpha *\rho, \qquad\rho(t,x)=\int_{\R^3} f^2(t,x,v)\,dv.\,\,\,\,\,
    \end{split}
\end{equation}
The corresponding Hartree equation is
\begin{equation}\label{Hartree}
    i\partial_t u+\frac{1}{2}\Delta u=\phi u, \qquad \phi=V_\alpha*|u|^2.
\end{equation}
We also consider the Vlasov--Hartree system
\begin{subequations}\label{VH}
\begin{align}
    &\partial_t f+v\cdot\nabla_x f+F\cdot\nabla_v f=0,
    \qquad F(t,x)=-\nabla_x\Psi(t,x), \qquad \Psi=V_\alpha*|u|^2,
    \label{eq:vlasov}\\
    &\,\,\,\,\,\,\,\,\, \,\,i\partial_tu+\frac{1}{2}\Delta_x u=\phi u, \qquad\phi=V_\alpha*\rho, \qquad \rho(t,x)=\int_{\mathbb{R}^3}f^2(t,x,v)\,dv. \label{eq:schrodinger}
\end{align}
\end{subequations}
Here $f=f(t,x,v)$ is real-valued, $u=u(t,x)$ is complex-valued, and $(x,v)\in\mathbb R^3\times\mathbb R^3$. All convolutions are in the spatial variable. Initial data are prescribed at $t=0$, with $f(0)=f_0$ and $u(0)=u_0$ for the components present in each model.

We use the square-root formulation of the kinetic equation \eqref{VR} and \eqref{eq:vlasov}: the physical phase-space density is $f^2$. Since $f$ and $f^2$ satisfy the same transport equation with the same force field, this formulation preserves the Vlasov dynamics and expresses the conserved mass as $\|f(t)\|_{L^2_{x,v}}^2$. It places the kinetic and wave components in a common $L^2$ setting, with spatial densities given by quadratic expressions in their respective amplitudes. 

The normalization $-\frac12\Delta_x$ makes the free Schr\"odinger Hamiltonian $|\xi|^2/2$ agree with the kinetic energy $|v|^2/2$. Thus frequency and velocity parametrize the same free rays $x=ty$. In \eqref{VR} and \eqref{Hartree}, each density generates its own interaction. In \eqref{VH}, the sources are exchanged: the bosonic density generates the force acting on the fermions, while the fermionic spatial density generates the potential acting on the bosons.

\medskip

The Vlasov--Riesz equation \eqref{VR} gives a kinetic description of a collisionless gas, whose particles evolve under the self-consistent force generated by their spatial density. The Hartree equation \eqref{Hartree} describes the mean-field quantum evolution of a Bose--Einstein condensate, with $u$ its macroscopic wave function and $|u|^2$ its spatial density.

In the Vlasov--Hartree system \eqref{VH}, $f^2$ describes a degenerate Fermi gas and $u$ a Bose--Einstein condensate. A coupling of this form was derived by Cárdenas, Miller, and Pavlović \cite{CMP25} through a mean-field approximation followed by a semiclassical limit for the fermions describing how a Bose-Einstein condensate interacts with a degenerate Fermi gas at zero temperature. 

\medskip

Our purpose is to establish forward small-data  modified scattering for all three models throughout the range \eqref{eq:Riesz_potential}, resolving a previously open problem for each of these models. The organizing observation is that their nonintegrable long-range dynamics admit a \emph{common description} in terms of \emph{an action} and \emph{a rescaled spatial density} along the free rays. The action produces a phase modulation for Hartree and a momentum translation for Vlasov--Riesz. The same equations determine the successive asymptotic corrections in both cases, and persist under the exchange of sources in Vlasov--Hartree.

\subsection{Long-range scattering and previous work} 
The long-range nature of these equations \eqref{VR}-\eqref{VH} comes from the interaction accumulated along the free evolution. Under the natural dispersive scaling, the spatial density has the form $t^{-3}m(t,x/t)$. The homogeneity of $V_\alpha$ then gives the corresponding potential of order $t^{-\alpha}$ on the spatial scale $x\sim t$. For Hartree, this potential accumulates in the \emph{phase} of the solution. For Vlasov, the force $F$ has size $t^{-1-\alpha}$, so particle velocities approach their limits with a correction of order $t^{-\alpha}$; this velocity correction accumulates in the positions of the particles. In both cases, the leading deviation from free evolution is therefore governed by 
\[
\int_1^t s^{-\alpha}\,ds.
\]
This integral diverges for $0<\alpha\leq1$, requiring a \emph{modification} of the free asymptotic dynamics. In particular, at the Coulomb exponent $\alpha=1$, it gives the familiar logarithmic correction. For $0<\alpha<1$, the leading term grows like $t^{1-\alpha}$.

\medskip

Modified scattering at the Coulomb exponent has a substantial history. For Hartree, Ginibre-Ozawa \cite{GO93} constructed modified wave operators, while Hayashi-Naumkin \cite{HN98} and Hayashi -Naumkin-Ozawa \cite{HNO98} established small-data modified scattering. Later proofs of small-data modified scattering use space-time resonances \cite{KP11}, wave packet testing \cite{VanHoose2024}, and a completely physical-space method \cite{PX}.

For Vlasov--Poisson, Choi-Kwon \cite{ChoiKwon2016} proved small-data modified scattering, and Ionescu, Pausader, Wang and Widmayer \cite{IPWW22} obtained explicit asymptotics using dispersive methods. Flynn, Ouyang, Pausader and Widmayer \cite{FOPW23} constructed modified wave operators and scattering maps. We further refer to \cite{PW21,PWY22,HPS24,BV26,HK26,KW25,SchlueTaylor2025,Big26,H26} for related advances and extensions in the asymptotic analysis of Vlasov-type systems.

\medskip

For $\alpha>1$, the potential along free rays is integrable in time; see \cite{NO92} for short-range Hartree scattering. For long-range interactions decaying more slowly than the Coulomb potential, the first correction leaves terms with time weight $t^{-2\alpha}$. These are integrable when $\alpha>1/2$. In this range, Hayashi-Naumkin \cite{HN01} obtained a modified-scattering formula for Hartree with an explicit power-law phase. Nakanishi \cite{Nak02b} developed a theory in which the initial data, asymptotic states, and convergence belong to the same weighted $L^2$ space for $1/2<\alpha<1$, complementing his Coulomb result \cite{Nak02a}.

Ginibre-Velo \cite{GV00I, GV00II} developed higher-order asymptotic expansions and constructed modified wave operators for Hartree throughout $0<\alpha\leq1$. Their analysis illustrates a central difficulty of this regime: corrections to the amplitude feed back into the phase, and the longer expansions required as $\alpha$ decreases demand increasingly precise control of spatial regularity. In the present work, we organize the forward dynamics through the coupled evolution of the \emph{action} and rescaled spatial \emph{density}, identifying this common structure and correction scheme for Hartree, Vlasov--Riesz, and Vlasov--Hartree.

On the kinetic side, Huang-Kwon \cite{HK24} established modified scattering and modified wave operators for $1-\delta<\alpha<1$. Hong-Pankavich \cite{HP26} extended forward modified scattering to the range $1/2<\alpha<1$ through a Lagrangian argument and proved small-data global existence for all $0<\alpha<1$. Their discussion of the remaining range $0<\alpha\leq1/2$ highlights a similar difficulty: controlling higher-order corrections requires sharper estimates on the spatial density and its derivatives.

\medskip

For Vlasov--Hartree, our earlier work \cite{HX} established modified scattering at the Coulomb exponent and for $1/2<\alpha<1$. Cavanagh \cite{Cav} later removed derivative assumptions on compactly supported fermionic data in the Coulomb case.

\medskip
In the range $0<\alpha\leq1/2$, forward small-data modified scattering remained open for all three models. The present work resolves this problem for all three models within a common action--density framework.

\medskip

Unlike the wave-operator problem \cite{GV00I, GV00II} where the asymptotic profile is prescribed, in the forward problem considered here, it must be recovered together with the necessary corrections from the evolving solution. For $0<\alpha\leq1/2$, the main difficulty lies in the feedback between the \emph{action} and the \emph{density} generating the field. The leading action produces a quadratic term of order $t^{-2\alpha}$, while the first correction to the rescaled spatial density changes the effective potential at the same order. Both contributions are nonintegrable and must be retained. Successive corrections require finer density asymptotics and corresponding derivative estimates, so optimal spatial decay alone does not determine the modified scattering behavior. We address this feedback through the coupled evolution of the \emph{action} and rescaled spatial \emph{density}, identifying a common structure and correction scheme for all three models.


\subsection{Main results}\label{subsec:main}
Following the approach of \cite{HX}, we describe transport and dispersion along the common free rays $x=ty$. The ray parameter $y$ plays the role of velocity for free transport and group velocity for the free Schr\"odinger evolution. We introduce the rescaled spatial densities
\begin{equation}\label{eq:m_rescaled}
    m_\hs(t,y):=t^3|u(t,ty)|^2,\qquad m_\ks(t,y):=t^3\rho(t,ty).
\end{equation}
In these coordinates, we construct a nonlinear action $\theta$ together with the solution. This action determines a phase modulation for the Hartree profile and a momentum translation for the kinetic profile. Its coupling with the rescaled spatial density has a common structure: for either component, the action and density satisfy equations of the form
\begin{subequations}\label{eq:action_density}
\begin{align}
    &\,\,\partial_t\theta_k-\frac{|\nabla_y\theta_k|^2}{2t^2} =t^{-\alpha}J(t)(V_\alpha*m_{j}), \label{eq:action}\\
    &\partial_t m_k-\frac{1}{t^2}\nabla_y \cdot(m_k\nabla_y\theta_k) =-\frac{1}{t^2}\nabla_y\cdot\mathcal J_k,\label{eq:density}
\end{align}
\end{subequations}
where the operator
\begin{equation}\label{eq:Jt}
    J(t):=(I-t^{-1}\Delta_y)^{-\frac{\alpha}{2}}=\beta_t(D), \qquad \beta_t(\xi):=(1+t^{-1}|\xi|^2)^{-\frac{\alpha}{2}}
\end{equation}
regularizes the potential, while $\mathcal{J}_k$ is the current associated with the corrected profile, with explicit formulas given in \eqref{eq:current}. Here, $k\in \{\hs,\ks\}$ labels the component whose profile is corrected by $\theta_k$, while $j\in \{\hs,\ks\}$ labels the source of its interaction field. The density $m_j$ drives the action $\theta_k$, whose gradient in turn governs the leading transport of $m_k$. For Hartree and Vlasov--Riesz, this feedback occurs within a single component, so $j=k$. For Vlasov--Hartree, it passes between the two components: the fermionic density $m_{\ks}$ drives the Hartree action $\theta_\hs$, while the bosonic density $m_\hs$ drives the kinetic action $\theta_\ks$. Thus, the same action--density structure accommodates both self-interaction and the exchange of sources. The model-dependent currents affect the limiting densities through the full evolution. Once these densities are fixed, the currents contribute only to the integrable action remainder, so the same recursion determines every nonintegrable correction.

\medskip

To describe the asymptotic corrections, set\footnote{$\left\lfloor {1}/{\alpha} \right\rfloor$ denotes the largest integer less than or equal to $1/\alpha$.}
\begin{equation}\label{eq:order}
    L:= \left\lfloor\frac{1}{\alpha}\right\rfloor-1, \qquad\text{so that }\,\, (L+1)\alpha\leq1<(L+2)\alpha,
\end{equation}
and define, for $\beta>0$ and $t\geq1$,
\begin{equation}\label{eq:time_factor}
    \Lambda_\beta(t):=\int_1^t s^{-\beta}\,ds
    =
    \begin{cases}
        \displaystyle\frac{t^{1-\beta}-1}{1-\beta},&\beta\neq1,\\[8pt]
        \log t,&\beta=1.
    \end{cases}
\end{equation}
The asymptotic action takes the form
\begin{equation}\label{eq:action_asymptotics}
    \Theta_k(t,y):=\sum_{\ell=0}^{L}\Lambda_{(\ell+1)\alpha}(t)\theta_{k,\ell}(y), \qquad k\in\{\hs,\ks\}.
\end{equation}
The coefficients $\theta_{k,\ell}$ are determined recursively by the limiting spatial densities $m_{k,\infty}$ through \eqref{eq:action_density}, starting with $\theta_{k,0}=V_\alpha*m_{j,\infty}$. Their explicit construction will be given in Section \ref{sec:asymptotics}. The choice of $L$ retains all nonintegrable time weights in the action equation so that the first omitted algebraic weight $t^{-(L+2)\alpha}$ is integrable. When $(L+1)\alpha=1$, the last term in \eqref{eq:action_asymptotics} is logarithmic; otherwise it is a power-law term. We later prove that $\theta_k-\Theta_k$ converges to a finite limit, which is absorbed into the scattering profile. 

\medskip
We also define the error function
\begin{equation}\label{eq:errorfunction}
 \mathfrak r_\alpha(t)=
 \begin{cases}
 t^{-\alpha}+t^{1-(L+2)\alpha},&(L+1)\alpha<1,\\
 t^{-\alpha}(1+\log t),&(L+1)\alpha=1.
 \end{cases}
\end{equation}
We note that $\mathfrak{r}_{\alpha}\rightarrow 0$ as $t\rightarrow \infty$ since $(L+2)\alpha>1$ and $\mathfrak{r}_{\alpha}\simeq t^{-\alpha}|1+\Lambda_{(L+1)\alpha}(t)|$.
\medskip

We now state the main results. We begin with the Hartree equation. The following theorem establishes global existence for small initial data and describes the asymptotic behavior through a finite expansion of the long-range phase.
\begin{theorem}\label{thm:Hartree}
Fix $0<\alpha\leq 1/2$. Let $N\in \N$ satisfy $ N\geq L+3$.  There exists a constant $\varepsilon_0>0$ depending only on $\alpha,\lambda,N$    such that the following holds. Suppose 
\begin{equation}\label{eq:initial-Hartree}
    \begin{split}
       \|u_0\|_{H^1}+\|\langle x\ra^N u_0\|_{L^2}\leq \varepsilon_0.
    \end{split}
\end{equation}
Then there exists a unique global solution $u\in C([0,\infty);H^{1}(\R^3))$ to \eqref{Hartree} with initial data $u(t=0)=u_0$. Moreover, there exist an asymptotic profile $h_{\infty}\in H^{1}\cap L^{\infty}$ and  an asymptotic action $\Theta_{\hs}(t,\cdot)\in W^{2,\infty}$ for fixed $t\geq 1$ defined in \eqref{eq:action_asymptotics} such that 
\begin{equation}\label{eq:asymptotic-hartree}
    \begin{split}
        \|u(t,x)-t^{-3/2}e^{i|x|^2/(2t)} e^{-i\Theta_{\hs}(t,x/t)}h_{\infty}\left(\frac{x}{t}\right)\|_{ L^{\infty}}\les \varepsilon_0 {\mathfrak{r}}_{\alpha}(t) t^{-3/2}.
    \end{split}
\end{equation}
\end{theorem}

\medskip
We next consider the Vlasov--Riesz equation. In this setting, the asymptotic action enters the scattering formula through its gradient, producing a translation of the kinetic profile in the free-ray variables.

\medskip

\begin{theorem}\label{thm:VR}
  Fix $0<\alpha\leq1/2$. Let $N\in \N$ satisfies $ N\geq L+3$. There exists a constant $\varepsilon_0>0$ depending only on $\alpha,\lambda,N$   such that the following holds. Suppose $f_0\in C^1_{x,v}$ and
\begin{equation}\label{eq:initial-VR}
    \begin{split}
        \sum_{|\mu|+|\nu|\leq N-1}
 \big\|\la x\ra^{2N+10}      \partial_x^\mu\partial_v^\nu f_0\big\|_{L^\infty\cap L^2}+\sum_{|\mu|+|\nu|= N}\big\|\la x\ra^{2N+10}       \partial_x^\mu\partial_v^\nu f_0\big\|_{L^2} \leq\varepsilon_0.
    \end{split}
\end{equation}
Then there exists a unique global solution $f\in C^1_{t,x,v}$ to \eqref{VR} with initial data $f(t=0)=f_0$. Moreover, there exist an asymptotic profile $f_{\infty}\in L^2_{x,v}\cap L_{x,v}^{\infty}$ and an asymptotic action $\Theta_{\ks}(t,\cdot)\in W^{2,\infty}$ for fixed $t\geq 1$ defined in \eqref{eq:action_asymptotics} such that
\begin{equation}\label{eq:asymptotic-VR}
    \begin{split}
   \left\|f(t,x,v)- f_\infty\left(\frac{x}{t},x-tv-\nabla_y\Theta_{\ks}\left(t,\frac{x}{t}\right)\right)\right\|_{L^2_{x,v}\cap L_{x,v}^{\infty}}\lesssim\varepsilon_0\mathfrak r_\alpha(t).
      \end{split}
\end{equation}
\end{theorem}

The global well-posedness assertions in Theorems \ref{thm:Hartree} and \ref{thm:VR} follow from Ginibre–Velo \cite[Propositions 2.1 and 2.3]{GVGWP} and Hong–Pankavich \cite[Theorem 2.1]{HP26}, respectively. In the kinetic case, the latter result is applied to the phase-space density $f_0^2$. Our analysis below focuses on establishing the profile estimates and modified-scattering asymptotics.

\medskip

Finally, we consider the coupled Vlasov--Hartree system. The common action--density formulation allows us to treat the two components within the same framework and establish their simultaneous modified scattering in the presence of the coupling.

\medskip

\begin{theorem}\label{thm:VH}
Fix $0<\alpha\leq 1/2$. Let $N\in \N$ satisfies $ N\geq L+3$. There exists a constant $\varepsilon_0>0$ depending only on $\alpha,\lambda,N$ such that the following holds. Suppose 
\begin{equation}\label{eq:initial-Hartree_VH}
    \begin{split}
       \|u_0\|_{H^{N-1}}+\|\langle x\ra^N u_0\|_{L^2}\leq \varepsilon_0,
    \end{split}
\end{equation}
and \eqref{eq:initial-VR} hold. \\
\hspace*{\parindent} Then there exists a unique global solution $f\in C^1_{t,x,v}$ and $u\in C_tH_x^1$ to \eqref{VH} with initial data $u(t=0)=u_0$ and $f(t=0)=f_0$. Moreover, there exist asymptotic profiles $h_{\infty}\in L^{\infty}\cap H^{1}$, $f_{\infty}\in L_{x,v}^{\infty}\cap L^2_{x,v} $ and asymptotic actions $\Theta_\hs(t,\cdot),\Theta_\ks(t,\cdot) \in W^{2,\infty}$ for fixed $t\geq 1$  such that 
\begin{equation}\label{eq:asymptotic-VH}
    \begin{split}
       \|u(t,x)-t^{-3/2}e^{i|x|^2/(2t)} e^{-i\Theta_\hs(t,x/t)}h_{\infty}\left(\frac{x}{t}\right)\|_{ L^{\infty}}&\les \varepsilon_0 t^{-3/2} {\mathfrak{r}}_{\alpha}(t),\\[.8em]
      \left\| f(t,x,v)-f_{\infty}\left(\frac{x}{t},x-tv-\nabla_y\Theta_\ks\left(t,\frac{x}{t}\right)\right)\right\|_{ L^2_{x,v}\cap L_{x,v}^{\infty}}&\les \varepsilon_0\mathfrak{r}_{\alpha}(t).
    \end{split}
\end{equation}

\end{theorem}

\begin{remark}\label{rem:regularity_VH}
The additional assumption $u_0\in H^{N-1}$ in Theorem \ref{thm:VH}, which is absent from Theorem \ref{thm:Hartree}, is needed only on the time interval $[0,1]$. For the Hartree equation, the small-time argument in Subsection \ref{subsec:boot_hartree} is carried out in the ray variable and uses only the Galilean norms $\|G(t)^\mu u(t)\|_{L^2}=\|\partial_y^\mu g(t)\|_{L^2}$, which as $t\to0$ reduce to the spatial weights in \eqref{eq:initial-Hartree} and involve no derivatives of $u_0$. In the coupled system, however, the kinetic component is driven by the physical density $|u|^2$. 
\end{remark}


\subsection{Organization}
The proof proceeds in two stages. We first construct the nonlinear actions together with the solutions and establish global control and convergence of the corrected profiles. Keeping the full nonlinear action at this stage incorporates all long-range orders into the bootstrap argument. We then use the limiting spatial densities to identify the finite action expansions and obtain the explicit scattering formulas.

After introducing the notation and preliminary estimates in Section \ref{sec:pre}, we explain the common structure in Section \ref{sec:idea}. We describe how the free-ray variables relate the kinetic and dispersive profiles, identify their common action--density evolution, and explain how this evolution generates the successive asymptotic corrections. This discussion provides the guiding ideas for the estimates and constructions that follow.

Section \ref{sec:boot} carries out the first stage of the proof. We begin with potential estimates that depend only on bounds for the rescaled density and can therefore be used for \emph{both} models. We then close the Hartree and Vlasov--Riesz bootstraps and establish convergence of their corrected profiles in separate subsections. Section \ref{sec:asymptotics} identifies the asymptotic actions: we first present the common recursive construction, and then establish the estimates needed to justify the expansions and control their remainders for each model. 

Finally, Section \ref{sec:VH} treats Vlasov--Hartree by combining the preceding estimates and expansions with exchanged sources. Each density drives the action of the other component, and the bootstrap and asymptotic construction close simultaneously. The coupled result thus illustrates how the common structure extends from the individual equations to their interaction.


\subsection*{Acknowledgment}
W. Huang thanks the warm hospitality of the Department of Mathematics at Yale University, where a substantial portion of this work was carried out.


\section{Notation and preliminaries}\label{sec:pre}

\subsection{Notation and function spaces}
We use the following notation. 
\begin{enumerate}
\item We write
\[
\langle x\rangle=(1+|x|^2)^{1/2},\qquad
\langle x,v\rangle=(1+|x|^2+|v|^2)^{1/2}.
\]
For a multi-index $\mu=(\mu_1,\ldots,\mu_d)\in\mathbb N_0^d$,
we write $|\mu|=\mu_1+\cdots+\mu_d$ and
$\partial_z^\mu=\partial_{z_1}^{\mu_1}\cdots\partial_{z_d}^{\mu_d}$.
When the variables are clear, we write $\partial^\mu$. For multi-indices $\nu\leq\mu$, where the inequality is understood componentwisely, we set
\[
\binom{\mu}{\nu}:=\prod_{j=1}^d\binom{\mu_j}{\nu_j}.
\]

\item
We use the unitary Fourier transform
\[
\widehat f(\xi)=(2\pi)^{-\frac{d}{2}} \int_{\R^d}e^{-ix\cdot\xi}f(x)\,dx.
\]
We use the standard Sobolev spaces $H^s(\mathbb R^d)$, with norm
\[
\|f\|_{H^s(\R^d)}
:=\left\|\langle\xi\rangle^s\widehat f(\xi)\right\|_{L^2_\xi},
\qquad \langle\xi\rangle:=(1+|\xi|^2)^{1/2}.
\]
For a nonnegative integer $N$, this norm is equivalent to
\[
\|f\|_{H^N(\mathbb R^d)}\simeq
\left(\sum_{|\mu|\leq N}
\|\partial_x^\mu f\|_{L^2(\mathbb R^d)}^2\right)^{1/2}.
\]
For a nonnegative integer $k$ and $1\leq p\leq\infty$, the space
$W^{k,p}(\mathbb R^d)$ consists of functions whose weak derivatives of
order at most $k$ belong to $L^p(\mathbb R^d)$. We use the norm
\[
\|f\|_{W^{k,p}(\mathbb R^d)}
:=\sum_{|\mu|\leq k}\|\partial_x^\mu f\|_{L^p(\mathbb R^d)}.
\]
Throughout, the norm on an intersection space is the sum of the two norms.

\item Unless otherwise specified, time is treated as a parameter, and
all norms and derivatives involve all the remaining variables of the
function. We omit the underlying domain and variable subscripts when
they are clear from the context. Thus, for $f=f(t,x,v)$ and
$\phi=\phi(t,x)$,
\[
\|f(t)\|_{H^N}
=\|f(t,\cdot,\cdot)\|_{H^N(\mathbb R^3_x\times\mathbb R^3_v)},
\qquad
\|\phi(t)\|_{H^N}=\|\phi(t,\cdot)\|_{H^N(\mathbb R^3_x)}.
\]
The same convention applies to $L^p$ and $W^{k,p}$ norms. Norms
involving only a subset of the variables, or different integrability
exponents in different variables, are indicated explicitly. 
In particular, $H_y^s(L_\eta^2)$ denotes the Sobolev space in $y$ with
values in $L^2(\R^3_\eta)$, with norm
\[
\|f\|_{H_y^s(L_\eta^2)} =\|\langle D_y\rangle^s f\|_{L^2_{y,\eta}}.
\]
For vector- and matrix-valued functions, norms are taken componentwise.

\item The notation $A\lesssim B$ means that $A\leq CB$ for a constant $C$ independent of $t$ and of the small parameters, which may change from line to line. Unless otherwise stated, constants may depend on $\alpha$, $\lambda$, and the fixed regularity indices. Dependence on additional parameters is indicated by a subscript. We write $A\lesssim_\delta B$, for example, when the constant may depend on $\delta$. We write $A\simeq B$ if $A\lesssim B$ and $B\lesssim A$.

The notation $a+$ denotes $a+\delta$ for an arbitrarily small fixed
$\delta>0$. If $X$ is a Banach space and $A=A(t)\geq0$, we write
\[
R(t)=\mathcal O_X(A(t))\quad\text{if}\quad \|R(t)\|_X\lesssim A(t),
\]
uniformly over the relevant range of $t$.

\item We use the notation $D=-i\nabla$. More generally, for a symbol
$m:\mathbb R^3\to\mathbb C$, we denote by $m(D)$ the corresponding
Fourier multiplier, defined by
\[
\widehat{m(D)f}(\xi)=m(\xi)\widehat f(\xi).
\]
We note that if $m(-\xi)=\overline{m(\xi)}$, then $m(D)$ preserves
real-valued functions; in particular, this holds whenever $m$ is
real-valued and even.

\item We write $[A,B]=AB-BA$ for the commutator of two operators.
\end{enumerate}


\subsection{Commutator and convolution estimates}
\begin{lemma}\label{lem:commutator}
Let $N\geq3$ be an integer. Suppose that $\nabla_y a\in H^{N-1}(\mathbb R^3)$ and
$f\in H^{N-1}(\mathbb R^3)$.
For $|\mu|\leq N$,
\begin{equation}\label{eq:commutator}
\big\|[\partial_y^{\mu},a]f\big\|_{L^2}
\lesssim\|\nabla_y a\|_{H^{N-1}}\|f\|_{H^{N-1}}.
\end{equation}
In particular, no $L^2$ bound on $a$ is required.
\end{lemma}
\begin{proof}
The Leibniz rule gives
\[
[\partial_y^\mu,a]f
=\sum_{0<\nu\leq\mu}\binom\mu\nu
(\partial_y^\nu a)(\partial_y^{\mu-\nu}f).
\]
For each summand, choose $j$ such that $\nu_j\geq1$. Since
\[
|\nu-e_j|+|\mu-\nu|=|\mu|-1\leq N-1,
\]
the differentiated Sobolev product estimate yields
\[
\|(\partial_y^\nu a)(\partial_y^{\mu-\nu}f)\|_{L^2}
=\|(\partial_y^{\nu-e_j}\partial_{y_j}a)
(\partial_y^{\mu-\nu}f)\|_{L^2}
\lesssim\|\partial_{y_j}a\|_{H^{N-1}}\|f\|_{H^{N-1}}.
\]
Here we use $N-1>3/2$. Summing over $\nu$ proves
\eqref{eq:commutator}.
\end{proof}

The same proof applies when $f$ takes values in $L^2_\eta$, with the Sobolev norms of $f$ taken in $H_y^{N-1}(L_\eta^2)$. More generally, for an integer $s\geq2$, the Sobolev product estimates in three dimensions give
\begin{equation}\label{eq:pre_hilbert_products}
\begin{aligned}
\|af\|_{H_y^s(L_\eta^2)}
&\lesssim\|a\|_{H_y^s}\|f\|_{H_y^s(L_\eta^2)},\\
\left\|\int_{\R^3}f(y,\eta)
\overline{g(y,\eta)}\,d\eta\right\|_{H_y^s}
&\lesssim\|f\|_{H_y^s(L_\eta^2)}\|g\|_{H_y^s(L_\eta^2)}.
\end{aligned}
\end{equation}
These follow by differentiating, applying H\"older's inequality in $y$ and, in the second estimate, Cauchy-Schwarz in $\eta$.

\begin{lemma}\label{lem:KP}
Let $r\geq1$, and write $\langle D\rangle^r=(1-\Delta)^{r/2}$.
Suppose that $a\in W^{1,\infty}(\mathbb R^3)$, with
$\nabla a\in H^{r-1}(\mathbb R^3)$, and
$f\in H^{r-1}(\mathbb R^3)\cap L^\infty(\mathbb R^3)$. Then
\begin{equation}\label{eq:KP}
\big\|[\langle D\rangle^r,a]f\big\|_{L^2}
\lesssim_r
\|\nabla a\|_{L^\infty}\|f\|_{H^{r-1}}
+\|\nabla a\|_{H^{r-1}}\|f\|_{L^\infty}.
\end{equation}
In particular, no $L^2$ assumption on $a$ is required. This estimate
also applies componentwise to vector- and matrix-valued functions.
\end{lemma}
This is the derivative refinement of the Kato--Ponce estimate \cite{KP}, in the form proved in \cite[Theorem 1.1]{Li}.

\medskip

We also use the fractional product estimate
\begin{equation}\label{eq:pre_fractional_product}
\|ab\|_{H^r} \lesssim_r\|a\|_{L^\infty}\|b\|_{H^r} +\|b\|_{L^\infty}\|a\|_{H^r}, \qquad r>0,
\end{equation}
for $a,b\in H^r(\mathbb R^3)\cap L^\infty(\R^3)$; see \cite{KP,Li}.

For the Riesz kernel $V_\alpha(y)=\lambda|y|^{-\alpha}$, we use
\[
\widehat{V_\alpha*f}(\xi) =c_\alpha\lambda|\xi|^{\alpha-3}\widehat f(\xi),
\]
for $f\in L^1\cap L^\infty$, in the sense of tempered distributions, where $c_\alpha>0$ depends on the Fourier normalization. In the range $0<\alpha\leq\tfrac12$ considered here, the following physical-space estimates also give two bounded derivatives of the potential.

\begin{lemma}\label{lem:pre_convolution}
Let $0<\alpha\leq\tfrac12$, and let $s\geq0$ be an integer. For $f\in W^{s,1}(\mathbb R^3)\cap W^{s,\infty}(\mathbb R^3)$ and $F\in L^1(\mathbb R^3;\mathbb R^3)\cap L^\infty(\mathbb R^3;\mathbb R^3)$, we have
\begin{align}
\|V_\alpha*f\|_{W^{s+2,\infty}}&\lesssim\|f\|_{W^{s,1}}+\|f\|_{W^{s,\infty}},
\label{eq:V_convolution}\\
\|\nabla V_\alpha*F\|_{W^{1,\infty}}& \lesssim\|F\|_{L^1}+\|F\|_{L^\infty}.
\label{eq:nabla_V_convolution}
\end{align}
For a vector field $F$, we write
\[
\nabla V_\alpha*F :=\sum_{\ell=1}^3(\partial_\ell V_\alpha)*F_\ell.
\]
\end{lemma}
\begin{proof}
Indeed, $V_\alpha$, $\nabla V_\alpha$, and $\nabla^2V_\alpha$ are integrable on the unit ball and bounded outside it. More precisely, for $|\beta|\leq2$ and $y\neq0$,
\[
|\partial^\beta V_\alpha(y)| \lesssim|y|^{-\alpha-|\beta|}.
\]
Since $\alpha+|\beta|<3$, these functions are locally integrable. Integration by parts outside a ball of radius $\delta$, followed by $\delta\to0$, shows that they are also the distributional derivatives of $V_\alpha$: the boundary terms vanish because $\alpha<1$. Consequently, for $q\in L^1\cap L^\infty$,
\begin{equation}\label{eq:pre_kernel_split}
\begin{aligned}
\|(\partial^\beta V_\alpha)*q\|_{L^\infty} &\leq
\|\partial^\beta V_\alpha\|_{L^1(|y|\leq1)}\|q\|_{L^\infty} +\|\partial^\beta V_\alpha\|_{L^\infty(|y|>1)}\|q\|_{L^1}\lesssim\|q\|_{L^1}+\|q\|_{L^\infty}.
\end{aligned}
\end{equation}
For $|\mu|\leq s+2$, choose $\beta\leq\mu$ with $|\beta|=\min\{2,|\mu|\}$. Then $|\mu-\beta|\leq s$, and
\[
\partial^\mu(V_\alpha*f) =(\partial^\beta V_\alpha)*(\partial^{\mu-\beta}f).
\]
This identity holds first for smooth compactly supported $f$ and extends to the stated spaces by cutoff and mollification. Applying \eqref{eq:pre_kernel_split} and summing over $\mu$ proves \eqref{eq:V_convolution}.
For \eqref{eq:nabla_V_convolution}, apply \eqref{eq:pre_kernel_split} componentwise to $(\partial_\ell V_\alpha)*F_\ell$ and $(\partial_j\partial_\ell V_\alpha)*F_\ell$, and sum over $j,\ell$.
\end{proof}

Finally, we record the elementary kernel bound used for the complementary force. Let $P_\ell=\psi(2^{-\ell}D)$, $\ell\in\Z$, where $\psi\in C_c^\infty(\R^3)$ is supported in $\{\tfrac12\leq|\xi|\leq2\}$ and
$\sum_{\ell\in\mathbb Z}\psi(2^{-\ell}\xi)=1$ for $\xi\neq0$. If $b_\ell\in C_c^4(\R^3)$ is supported in $\{2^{\ell-1}\leq|\xi|\leq2^{\ell+1}\}$, then
\begin{equation}\label{eq:pre_dyadic_kernel}
\|\mathcal F^{-1}b_\ell\|_{L^1} \lesssim \sum_{|\beta|\leq 4}2^{\ell|\beta|} \|\partial_\xi^\beta b_\ell\|_{L^\infty}.
\end{equation}
Indeed, after rescaling to $\ell=0$, integration by parts gives
\[
|\mathcal F^{-1}b_0(y)| \lesssim\langle y\rangle^{-4} \sum_{|\beta|\leq4}\|\partial_\xi^\beta b_0\|_{L^1} \lesssim\langle y\rangle^{-4} \sum_{|\beta|\leq4}\|\partial_\xi^\beta b_0\|_{L^\infty}.
\]
Integrating in $y$ and undoing the rescaling proves \eqref{eq:pre_dyadic_kernel}.


\subsection{Galilean vector fields}
For the free Schr\"odinger operator $i\partial_t+\tfrac12\Delta_x$, define
\[
G_j(t):=x_j+it\partial_{x_j},\qquad G(t)=(G_1(t),G_2(t),G_3(t)).
\]
These operators commute with one another and all commute with the free Schr\"odinger operator
\[
\left[i\partial_t+\tfrac12\Delta_x,\,G_j(t) \right]=0,\qquad G(0)=x.
\]
Moreover, we have the useful expansion of the Galilean vector field,
\begin{equation*}
    \begin{split}
        G_j(t)=e^{i\frac{|x|^2}{2t}}(it\pr_{x_j})e^{-i\frac{|x|^2}{2t}}.
    \end{split}
\end{equation*}
As a consequence, for any multi-index $\mu$, we have
\begin{equation}\label{eq:G_expression}
\begin{split}
        G(t)^{\mu}=e^{i|x|^2/2t}(it)^{|\mu|}\pr_x^\mu e^{-i|x|^2/2t}.
    \end{split}
\end{equation}

\begin{lemma}\label{lem:galilean_profile}
Let $t>0$, let $u(t)\in L^2(\R^3)$, and let $g(t,y)=t^{3/2}e^{-it\frac{|y|^2}{2}}u(t,ty)$ be the profile defined in \eqref{eq:g}. Then, for every multi-index $\mu$,
\begin{equation}\label{eq:G_profile_identity}
    \pr_y^{\mu}g(t,y)=(-i)^{|\mu|}\,t^{\frac32}\,e^{-it\frac{|y|^2}{2}}\,\big(G(t)^{\mu}u(t)\big)(ty)
\end{equation}
in the sense of distributions. In particular,
\begin{equation}\label{eq:G_profile_L2}
\|\pr_y^{\mu}g(t)\|_{L^2_y}=\|G(t)^{\mu}u(t)\|_{L^2_x}.
\end{equation}
\end{lemma}

\begin{proof}
By definition of $g$, $e^{-i\frac{|x|^2}{2t}}u(t,x)=t^{-\frac32}g(t,x/t)$, so the chain rule gives
\[
\pr_x^{\mu}\Big[e^{-i\frac{|x|^2}{2t}}u(t,x)\Big]=t^{-\frac32-|\mu|}\,(\pr_y^{\mu}g)(t,x/t).
\]
Inserting this into \eqref{eq:G_expression} yields
\[
\big(G(t)^{\mu}u\big)(t,x)=e^{i\frac{|x|^2}{2t}}(it)^{|\mu|}\,t^{-\frac32-|\mu|}\,(\pr_y^{\mu}g)(t,x/t)
= i^{|\mu|}\,t^{-\frac32}\,e^{i\frac{|x|^2}{2t}}\,(\pr_y^{\mu}g)(t,x/t).
\]
Setting $x=ty$ and solving for $\pr_y^\mu g$ gives \eqref{eq:G_profile_identity}. Finally, \eqref{eq:G_profile_L2} follows from \eqref{eq:G_profile_identity} by the change of variables $x=ty$, $dx=t^{3}dy$.
\end{proof}


\subsection{Hamiltonian transport}
Let $\mathcal{H}=\mathcal{H}(t,x,v)$  be a real-valued Hamiltonian and set
\[
b_{\mathcal{H}}:=\left(\nabla_v \mathcal{H},-\nabla_x \mathcal{H}\right).
\]
The vector field $b_\mathcal{H}$ is divergence-free on phase space. Its characteristic flow therefore preserves Lebesgue measure.  For functions on phase space, we use the Poisson bracket
\begin{equation}\label{eq:poisson}
   \{f,g\}=\nabla_x f\cdot \nabla_vg-\nab_v f\cdot \nab_x g.
\end{equation}
\begin{lemma}\label{lem:hamiltonian_transport}
    Let $1\leq r\leq \infty$, and suppose that
    \begin{equation}\label{eq:hamiltonian_transport_eqn}
        \partial_t f +\{ f,\mathcal{H}\}=g,\quad f(t=0)=f_0, 
    \end{equation}
    on $[0,T]$. Assume that the characteristic flow of $b_\mathcal{H}$ is well-defined and that $g\in L^1([0,T];L^r(\mathbb R^3\times \R^3))$. Then 
    \begin{equation*}
    \begin{split}
        \|f(t)\|_{L^r_{x,v}}\leq \|f_0\|_{L^r_{x,v}}+\int_0^t \|g(s)\|_{L^r_{x,v}}ds, \quad t\in [0,T].
    \end{split}
\end{equation*}
\end{lemma}
\noindent The proof follows by integrating \eqref{eq:hamiltonian_transport_eqn} along the Hamiltonian flow and using Liouville’s theorem. We refer to Glassey \cite{Glassey} for the standard characteristic theory for Vlasov equations. 

We will also frequently use the Leibniz rule for the Poisson bracket and its commutation with differentiation:
\begin{equation*}
    \{fg,h\}=f\{g,h\}+g\{f,h\},
    \qquad
    \nabla\{f,g\}=\{\nabla f,g\}+\{f,\nabla g\}.
\end{equation*}


\section{A common structure for long-range dynamics}\label{sec:idea}
In this section, we discuss the \emph{common action--density} structure which has two roles in our argument. It determines the corrections used to construct limiting profiles, and it generates the successive terms in the asymptotic action. We explain these roles by placing the wave and kinetic equations in \emph{common coordinates}, identifying the \emph{action} that \emph{corrects} their evolution, and examining the \emph{resulting density flux}. We defer the global bounds and asymptotic estimates to Sections \ref{sec:boot} and \ref{sec:asymptotics}.

\subsection{Profiles along free rays}
For $t\geq1$, we write $y=x/t$ and introduce the profiles\footnote{These profiles correspond to the pseudoconformal change of variables, with the original time variable retained.}
\begin{equation}\label{eq:g}
    \begin{split}
        g_\hs(t,y) :=t^{3/2}e^{-it\frac{|y|^2}{2}}u(t,ty), \qquad g_\ks(t,y,\eta) :=f\left(t,ty,y+\frac{\eta}{t}\right).
    \end{split}
\end{equation}
{The Hartree transformation removes the quadratic phase and the spatial decay of the free Schr\"odinger evolution. In the kinetic transformation, $\eta=tv-x$ is constant along free characteristics. The map $(y,\eta)\mapsto(x,v)=(ty,y+\eta/t)$ is symplectic,  and in particular preserves phase-space volume. Thus both transformations preserve the corresponding $L^2$ norms.}

In these coordinates, the rescaled spatial densities \eqref{eq:m_rescaled} take the form
\begin{equation}
    m_{\hs}(t,y)=|g_{\hs}(t,y)|^2, \qquad
m_\ks(t,y)=\int_{\R^3}g_\ks^2(t,y,\eta)\,d\eta.
\end{equation}
Thus both rescaled densities are quadratic expressions in the profiles at fixed $y$, with momentum integrated out in the kinetic case. This common form allows us to derive their continuity equations in parallel. 

We retain the component and source labels from Subsection \ref{subsec:main}: $k\in\{\hs,\ks\}$ denotes the component under consideration, and $j\in\{\hs,\ks\}$ denotes the source of its interaction field. Set 
\begin{equation}\label{eq:varphi}
    \varphi_k:=V_\alpha*m_j.
\end{equation}
For the scalar equations \eqref{VR}--\eqref{Hartree}, we have $k=j$, since each component generates its own interaction field. In the coupled system \eqref{VH}, $k\neq j$, since each component generates the interaction field acting on the other. Homogeneity of $V_\alpha$ then gives
\begin{equation}
    \phi(t,ty)=t^{-\alpha}\varphi_{\hs}(t,y), \qquad \Psi(t,ty)=t^{-\alpha}\varphi_\ks(t,y)
\end{equation}
in \eqref{Hartree} and \eqref{VR}, respectively. The corresponding profile equations are therefore
\begin{equation}\label{eq:g_H}
    i\partial_tg_\hs+\frac{1}{2t^2} \Delta_yg_\hs=t^{-\alpha}\varphi_\hs g_\hs,
\end{equation}
and 
\begin{equation}\label{eq:g_K}
    \partial_tg_\ks+\frac{\eta}{t^2}\cdot\nabla_yg_\ks-t^{-\alpha}\nabla_y \varphi_\ks\cdot\nabla_\eta g_\ks=0.
\end{equation}
These two profile equations have the same Hamiltonian expression,
\[
\mathcal H_k(t,y,\eta) =\frac{|\eta|^2}{2t^2}+t^{-\alpha}\varphi_k(t,y), \quad k\in\{\hs,\ks\}.
\]
It generates the kinetic characteristic flow, and replacing $\eta$ by $-i\nabla_y$ gives the Hartree operator.

This formulation also makes the interaction scale explicit. Along kinetic characteristics,
\begin{equation}
    \begin{split}
        &\dot y=t^{-2}\eta,\\
        &\dot\eta=tF(t,ty)=-t^{-\alpha}\nabla_y\varphi_\ks.
    \end{split}
\end{equation}
The additional spatial decay of the physical force is therefore compensated by the factor $t$ in the evolution of $\eta$. Both profile equations contain an interaction with time weight $t^{-\alpha}$, while the remaining free evolution carries $t^{-2}$, as seen in \eqref{eq:g_H} and \eqref{eq:g_K}. The common Hamiltonian suggests correcting this long-range interaction through an action on the ray variable $y$.


\subsection{The common nonlinear action correction}\label{subsec:correction}
Let $\theta_\hs$ and $\theta_\ks$ be real-valued functions, to be chosen below, and define
\begin{equation}\label{eq:h}
    h_\hs:=e^{i\theta_\hs}g_\hs, \qquad h_\ks(t,y,\eta):=g_\ks\left(t,y,\eta-\nabla_y\theta_\ks(t,y)\right).
\end{equation}
The action acts by phase modulation on the wave and by momentum translation on the kinetic profile. Both transformations preserve the spatial density:
\begin{equation}
    m_\hs(t,y)=|h_\hs(t,y)|^2, \qquad m_\ks(t,y)=\int_{\R^3} h_\ks^2(t,y,\eta)\,d\eta.
\end{equation}
Since the displacement is a gradient, this kinetic translation is also symplectic. These properties allow the action to modify the evolution while keeping the source density in the same quadratic form. 

The equation for the action follows from the \emph{transformed Hamiltonian} for the modified profiles $h_k$. Suppressing the component label temporarily, it is
\begin{equation}
    \mathcal K(t,y,\eta)=\frac{|\eta-\nabla_y\theta|^2}{2t^2}+t^{-\alpha} \varphi(t,y)-\partial_t\theta(t,y).
\end{equation}
This Hamiltonian describes both corrected evolutions. The kinetic momentum translation replaces $\eta$ by $\eta-\nabla_y\theta$, while the Hartree phase modulation makes the corresponding replacement of the momentum operator, from $-i\nabla_y$ to $-i\nabla_y-\nabla_y\theta$. Both corrections therefore produce the same scalar term 
\[
\frac{|\nabla_y\theta|^2}{2t^2} +t^{-\alpha} \varphi-\partial_t\theta.
\]
The choice $\partial_t\theta=t^{-\alpha}\varphi$ cancels the interaction potential but leaves the quadratic contribution generated by the momentum shift. Since the leading action gradient has size $t^{1-\alpha}$, this contribution has size $t^{-2\alpha}$ and is nonintegrable for $0<\alpha\leq\frac12$. Cancelling both terms leads to the Hamilton-Jacobi equation
\[
\partial_t\theta-\frac{|\nabla_y\theta|^2}{2t^2}=t^{-\alpha}\varphi.
\]
This equation already appears in Ginibre-Velo’s analysis of Hartree \cite{GV00I, GV00II}. Here, we observe that it also determines the kinetic correction through the same Hamiltonian structure. 

\medskip

For the global estimates, we use a regularized version of this equation. The coefficients in the modified profile equations involve spatial derivatives of the action, so their control depends on the regularity of the interaction potential. We separate this potential into a smoother part incorporated into the action and a complementary part retained in the profile evolution. Specifically, suppressing the component label, we define 
\begin{equation}\label{eq:decomposition}
    \varphi_{\lo}:=J(t)\varphi, \qquad
    \varphi_{\hi}:=(\operatorname{Id}-J(t))\varphi, \qquad J(t)=(I-t^{-1}\Delta_y)^{-\frac{\alpha}{2}}.
\end{equation}
The multiplier operator $J(t)$ is close to the identity at frequencies $|\xi|\ll t^\frac{1}{2}$, and gains $\alpha$ derivatives at higher frequencies, at the cost of a time factor $t^{\frac{\alpha}{2}}$. Thus it preserves the low-frequency potential while providing additional regularity for the action estimates. On the other hand, the complementary part gains time decay. Indeed, the symbol of $I-J(t)$ satisfies
\begin{equation}\label{eq:beta_bound}
    0\leq 1-\beta_t(\xi) \lesssim_\alpha \min\left\{1,t^{-1}|\xi|^2\right\},
\end{equation}
because
\[
    1-\beta_t(\xi)=\frac\alpha2\int_0^{t^{-1}|\xi|^2}
    (1+s)^{-1-\frac\alpha2}\,ds\,\leq \, \frac{\alpha}{2}\int_0^{t^{-1}|\xi|^2}1\,ds\lesssim_\alpha\,t^{-1}|\xi|^2.
\]
At low frequencies, this supplies a factor $t^{-1}$; at high frequencies, the decay follows from the regularity of the source density and the smoothing of the Riesz potential. The estimates in Subsection \ref{subsec:boot_common} show that $t^{-\alpha}\varphi_\hi$ is integrable in time; hence, this decomposition gains regularity in the action while leaving only an integrable field outside it.

\medskip

We accordingly define the action by 
\begin{equation}\label{eq:theta}
    \partial_t\theta-\frac{|\nabla_y\theta|^2}{2t^2} =t^{-\alpha}\varphi_{\mathrm{low}}, \qquad \theta(1,y)=0.
\end{equation}
With this choice, the transformed Hamiltonian consists of the free quadratic term, the mixed term $-t^{-2}\eta\cdot\nabla_y\theta$, and the residual potential $t^{-\alpha}\varphi_\hi$. Restoring the component labels, set
\begin{equation}\label{eq:Gamma}
    \Gamma_k:=\nabla_y\theta_k.
\end{equation}
For the Hartree component, define
\begin{equation}\label{eq:T_theta}
    \mathcal{T}_{\theta}:=\Gamma_\hs\cdot \nabla_y+\frac{1}{2}\nabla_y\cdot \Gamma_\hs,
\end{equation}
and this first-order operator is skew-adjoint on $L^2$ from integration by parts:
\[
  \Re\la\,\mathcal T_\theta\, w,\,w\ra_{L^2}=0.
\] 
The modified profiles satisfy
\begin{equation}\label{eq:PDE_h_H}
    i\partial_t h_\hs +\frac{1}{2t^2}\Delta_y h_\hs-\frac{i}{t^2}\mathcal T_\theta h_\hs=t^{-\alpha}\varphi_{\hs,\hi}h_\hs,
\end{equation}
and 
\begin{equation}\label{eq:PDE_h_K}
    \partial_t h_\ks+\frac{\eta-\Gamma_\ks}{t^2}\cdot\nabla_yh_\ks+\left(\frac{\nabla\Gamma_\ks\eta}{t^2}-t^{-\alpha} \nabla_y\varphi_{\ks,\hi}\right)\cdot\nabla_\eta h_\ks=0.
\end{equation}
Both equations retain the conservation structure needed for the energy estimates. In particular, on the kinetic side, it is still a Hamiltonian system with Hamiltonian $\K$ defined in \eqref{eq:PDE_h_Hamiltonian}.

For solutions of size $\varepsilon$, the expected action bounds are
\[
\|\Gamma_k(t)\|_{L^\infty}+\|\nabla_y \Gamma_k(t)\|_{L^\infty}\lesssim\varepsilon^2t^{1-\alpha}, \qquad k\in \{\hs,\ks\}.
\]
Thus $t^{-2}\Gamma_k$ and $t^{-2}\nabla\Gamma_k$ have the integrable time weight $t^{-1-\alpha}$. The complementary field is likewise integrable in the norms needed for the corrected evolution. These rates must be combined with bounds on spatial derivatives for Hartree, and with momentum moments and mixed derivatives for Vlasov--Riesz.  

Section \ref{sec:boot} carries out these estimates and constructs the limiting profiles. {For Hartree, a difference energy estimate gives convergence with one derivative lost, and interpolation gives convergence at every Sobolev order below $N$.} For the kinetic equation, the corresponding argument uses the weighted control of the transport terms. 


\subsection{The action--density system and its asymptotics}\label{subsec:feedback}
The action equation determines the correction through the evolving density. To identify its asymptotic coefficients, we also need to understand how the action changes that density. Define the residual currents
\begin{equation}\label{eq:current}
    \mathcal J_\hs:=\Im(\overline{h_\hs} \nabla_y h_\hs),\qquad \mathcal J_\ks:=\int_{\R^3}\eta\,h_\ks^2(t,y,\eta)\,d\eta.
\end{equation}
Computing the density evolution in each corrected equation gives the same continuity law:
\begin{equation}\label{eq:continuity-m}
    \partial_tm_k -\frac1{t^2}\nabla_y\cdot(m_k\nabla_y \theta_k)=-\frac{1}{t^2}\nabla_y \cdot\mathcal J_k, \qquad k\in\{\hs,\ks\},
\end{equation}
which agrees with \eqref{eq:density}. Together with the action equation \eqref{eq:theta}, this is the action--density system introduced in \eqref{eq:action_density}. The density supplies the potential driving the action, while the action gradient determines the leading density flux. The residual current records the part of this flux that depends on the particular model. 

The two fluxes have different asymptotic roles. The uniform Sobolev and momentum-weighted bounds for the modified profiles to be established in Section \ref{sec:boot} give
{
\[
\sup_{s\geq1}
\|\mathcal J_k(s)\|_{W^{N-1,1}_y\cap H^{N-1}_y}
\lesssim\varepsilon^2.
\]}
Since the residual-current term carries the factor $s^{-2}$, it gives a density contribution of order $t^{-1}$. After convolution with $V_\alpha$ and multiplication by $t^{-\alpha}$, its contribution to the action equation therefore has the integrable weight $t^{-1-\alpha}$. 

The action-driven flux, on the other hand, produces a slower density correction. Indeed, the leading growth $\Gamma_k=\mathcal O(t^{1-\alpha})$ gives a term of order $t^{-1-\alpha}$ in the density equation, which after integrating from $t$ to infinity is of order $t^{-\alpha}$. This density correction modifies the interaction potential and therefore contributes at order $t^{-2\alpha}$ to the action equation. This is the \emph{action--density feedback}: the action \emph{determines a density correction}, which in turn \emph{contributes to the next order of the action}. Its contribution has the same order as the quadratic action term, so \emph{both} must enter the next action coefficient. 

\medskip

We illustrate how this feedback works with the first correction for the scalar equations Hartree and Vlasov--Riesz (so $k=j$). Let $m_{k,\infty}$ denote the density associated with the limiting modified profile. With the convention in \eqref{eq:action_asymptotics}, we define
\begin{equation}\label{eq:zero_correction_action}
    \theta_{k,0}:=V_\alpha*m_{j,\infty},\qquad \Gamma_{k,0}:= \frac{\nabla_y\theta_{k,0}}{1-\alpha}.
\end{equation}
The leading action gradient is then $t^{1-\alpha}\Gamma_{k,0}$. To determine the corresponding density correction, we integrate the density equation towards the limiting density:
\begin{align*}
    m_k(t)-m_{k,\infty}=-\int_t^\infty \frac{1}{s^2} \nabla_y\cdot\left(m_k(s)\Gamma_k(s)\right)\,ds+\int_t^\infty \frac{1}{s^2}\nabla_y\cdot\mathcal J_k(s)\,ds.
\end{align*}
As discussed before, the residual-current integral is of smaller order and does not enter the correction. Substituting $m_{k,\infty}$ and $s^{1-\alpha}\Gamma_{k,0}$ into the first integral identifies the first density term $t^{-\alpha}m_{k,1}$, where
\[
m_{k,1}=-\frac1\alpha\nabla_y\cdot \left(m_{k,\infty}\Gamma_{k,0}\right).
\]
We now insert the density correction into the action equation. The correction to the source potential contributes $t^{-2\alpha}V_\alpha*m_{j,1}$, while the leading action gradient contributes half of $t^{-2\alpha}|\Gamma_{k,0}|^2$. Thus the next action coefficient is
\[
\theta_{k,1}=V_\alpha*m_{k,1}+\frac12|\Gamma_{k,0}|^2.
\]
The first term records the change in the source density, and the second records the quadratic effect of the action itself. The corresponding correction to the action is
\[
\Lambda_{2\alpha}(t)\theta_{k,1},
\qquad
\Lambda_{2\alpha}(t)=\int_1^t s^{-2\alpha}\,ds.
\]
It grows like $t^{1-2\alpha}$ when $\alpha<\frac12$ and becomes logarithmic at $\alpha=\frac12$. 

This computation reveals the order of the general construction for the correction. Previously obtained action and density coefficients \emph{first} determine the \emph{next density coefficient} through the flux $m_k\Gamma_k$. The \emph{new source-density} coefficient then combines with \emph{quadratic products} of the known action gradients to determine the \emph{next action} coefficient. Repeating this procedure produces the successive time weights $t^{-\alpha},t^{-2\alpha},t^{-3\alpha},\dots$ in the action equation. We retain every nonintegrable weight, \emph{stopping} once the remaining terms are integrable. When $1/\alpha$ is an integer, the last retained weight is $t^{-1}$, and its time integral gives the logarithmic term in \eqref{eq:action_asymptotics}. 

This construction also explains the role of the model-dependent currents $\mathcal{J}_k$. The residual currents and the complementary potential $\varphi_{k,\hi}$ influence the limiting densities through the full evolution. Once these densities are fixed, their remaining contributions enter only the integrable remainder in the action equation.

\medskip

To justify this construction, the density remainders must be controlled before convolution with $V_\alpha$: at the next step, they reappear in the product $m_k\Gamma_k$. Section \ref{sec:asymptotics} formulates the recursion and establishes these density estimates together with the corresponding action remainder estimates. This separates the identification of the coefficients from the derivative bounds needed to validate the expansion. 

After all nonintegrable terms have been extracted, let $\Theta_k$ denote the resulting finite action expansion. The remainder estimates give
\[
\theta_k(t)-\Theta_k(t)\quad\longrightarrow\quad \theta_{k,\mathrm{final}},
\]
with convergence of the spatial gradient as needed for the kinetic correction. This finite limit is absorbed into the limiting profile by a phase modulation for Hartree and a momentum translation for Vlasov--Riesz. 

\medskip

For Vlasov--Hartree, the same construction applies with $j\neq k$. At each order, the two density coefficients are obtained from the already determined coefficients of their respective fluxes. Each then supplies the source for the next action coefficient of the other component. The exchange of sources therefore preserves the recursive structure, allowing the coupled asymptotics to be constructed from the two limiting densities. Section \ref{sec:VH} combines this construction with the common estimates to establish modified scattering for the coupled system.


\section{Global bounds and profile convergence}\label{sec:boot}
In this section, we establish global bounds and convergence of the corrected profiles \eqref{eq:h} for the Hartree equation \eqref{Hartree} and the Vlasov–Riesz equation \eqref{VR}. We first prove the field and action estimates common to these two scalar equations, and then treat their profile estimates and asymptotics separately. The Vlasov–Hartree system \eqref{VH} will be considered in Section \ref{sec:VH}, where these arguments are applied with exchanged sources.

Let $N\geq 4$ be an integer, whose value in terms of $\alpha$ will be chosen in Section \ref{sec:asymptotics}. We suppress the component labels when considering either scalar equation. All spatial norms and derivatives are in the ray variable $y$, unless otherwise indicated. Implicit constants may depend on $\alpha$, $N$, and $\lambda$. 

\subsection{Common field and action estimates}\label{subsec:boot_common}
For either scalar equation, let $m$ denote its rescaled spatial density and write 
\[
\varphi=V_\alpha*m, \qquad \Gamma=\nabla_y\theta.
\]
The potential decomposition and the action are those defined in \eqref{eq:decomposition} and \eqref{eq:theta}. We begin the bootstrap argument at $t=1$. The short-time estimates needed to initialize the bootstrap will be discussed separately for the Hartree and Vlasov–Riesz equations later.

On a time interval $[1,T]$, we impose the bootstrap assumptions
\begin{equation}\label{boot}
\|m(t)\|_{H^N}\leq\varepsilon^2,\qquad\|\nabla_y \Gamma(t)\|_{H^{N+1-\alpha}}\leq\varepsilon^2t^{1-\alpha},
\end{equation}
where $0<\varepsilon_0<\varepsilon\leq1$ and $\varepsilon_0$ is the bound for the initial data in \eqref{eq:initial-Hartree} or \eqref{eq:initial-VR}. In particular, the conservation of mass for both scalar equations gives
\begin{equation}\label{eq:mass_conservation}
    \|m(t)\|_{L^1}\leq \varepsilon_0^2.
\end{equation}
We first estimate the potential in terms of the density.
\begin{lemma}\label{lem:est_varphi}
    For $t\geq 1$, and $m(t)\in L^1\cap H^N$, we have
    \begin{equation}\label{eq:nabla_varphi}
        \|\nabla_y\varphi(t)\|_{W^{N-1,\infty}}+\|\nabla_y\varphi_{\lo}(t)\|_{L^\infty}\lesssim\|m(t)\|_{L^1}+\|m(t)\|_{H^N},
    \end{equation}
    \begin{equation}\label{eq:nabla2_varphi}
        \|\nabla_y^2\varphi(t)\|_{H^{N+1-\alpha}}+\|\nabla_y^2\varphi_\lo(t)\|_{H^{N+1-\alpha}}\lesssim \|m(t)\|_{L^1}+\|m(t)\|_{H^N},
    \end{equation}
    and
    \begin{equation}\label{eq:gain_varphi_low}
    \|\nabla_y^2\varphi_\lo(t)\|_{H^{N+1}} \lesssim t^{\frac{\alpha}{2}}\left(\|m(t)\|_ {L^1} +\|m(t)\|_{H^N} \right).
    \end{equation}
The complementary potential satisfies
\begin{equation}\label{eq:varphi_hi}
    \|\varphi_{\hi}(t)\|_{H^N}\lesssim t^{-1}\left(\|m(t)\|_{L^1}+\|m(t)\|_{H^N} \right),
\end{equation}
and
\begin{equation}\label{eq:varphi_hi_derivative}
    \|\nabla_y\varphi_\hi(t)\|_{H^N} \lesssim t^{-1+\frac{\alpha}{2}}\|m(t)\|_{H^N}.
\end{equation}
The bounds for the unregularized potential $\varphi$ in \eqref{eq:nabla_varphi} and \eqref{eq:nabla2_varphi} hold for every $t\geq 0$.
\end{lemma}
\begin{proof}
    The Fourier representation of the Riesz potential $V_\alpha$ gives
\[
\widehat\varphi(\xi)=c_\alpha\lambda|\xi|^{\alpha-3}\widehat m(\xi).
\]
At frequencies $|\xi|\leq1$, we use
$|\widehat m(\xi)|\lesssim\|m\|_{L^1}$. At frequencies $|\xi|>1$, we use the $H^N$ norm. In particular, Fourier inversion and Cauchy–Schwarz give
\begin{align*}
    \|\nabla_y\varphi(t)\|_{W^{N-1,\infty}}&\lesssim\|m(t)\|_{L^1}\int_{|\xi| \leq 1}|\xi|^{\alpha-2}\,d\xi+\int_{|\xi|>1} |\xi|^{\alpha-3}\langle\xi\rangle^N|\widehat m(t,\xi)|\,d\xi\\
    &\lesssim \|m(t)\|_{L^1} +\|m(t)\|_{H^N}.
\end{align*}
The same argument applies to $\nabla_y\varphi_\lo$, since $0<\beta_t\leq1$. Similarly,
\begin{align*}
    \|\nabla_y^2\varphi(t)\|_{H^{N+1-\alpha}}^2 &\lesssim\|m(t)\|_{L^1}^2 \int_{|\xi|\leq1}|\xi|^{2\alpha-2}\,d\xi+\int_{|\xi|>1}\langle\xi\rangle^{2N}|\widehat m(t,\xi)|^2\,d\xi\\
    &\lesssim\|m(t)\|_{L^1}^2+\|m(t)\|_{H^N}^2.
\end{align*}
This also bounds the low-frequency part. The additional regularity follows from trading regularity for time growth:
\begin{align*}
    \langle\xi\rangle^\alpha \beta_t(\xi)=\left(\frac{1+|\xi|^2}{1+t^{-1}|\xi|^2} \right)^{\frac{\alpha}{2}}=t^{\frac{\alpha}{2}}\left(\frac{1+|\xi|^2}{t+|\xi|^2}\right)^{\frac{\alpha}{2}} \leq t^{\frac{\alpha}{2}}, \qquad t\geq1.
\end{align*}
Hence, Plancherel’s theorem and \eqref{eq:nabla2_varphi} give
\begin{align*}
    \|\nabla_y^2\varphi_\lo(t)\|_{H^{N+1}}&=\left\|\langle\xi\rangle^\alpha\beta_t(\xi)\, \langle\xi\rangle^{N+1-\alpha}\widehat{\nabla_y^2\varphi}(t,\xi) \right\|_{L^2_\xi}\leq t^{\frac{\alpha}{2}}\|\nabla_y^2\varphi(t)\|_{H^{N+1-\alpha}}\\ 
    &\lesssim t^{\frac{\alpha}{2}} \left(\|m(t)\|_{L^1}+\|m(t)\|_{H^N}\right).
\end{align*}
For the complementary potential, \eqref{eq:beta_bound} yields
\begin{align*}
    \|\varphi_\hi(t)\|_{H^N}^2&\lesssim t^{-2}\int_{\R^3}\langle\xi\rangle^{2N} |\xi|^{2\alpha-2}|\widehat m(t,\xi)|^2\,d\xi\lesssim t^{-2}\left(\|m(t)\|_{L^1}^2+\|m(t)\|_{H^N}^2\right).
\end{align*}
Here the low-frequency singularity is square-integrable, while
$|\xi|^{\alpha-1}\leq1$ for $|\xi|\geq1$. Finally, interpolating \eqref{eq:beta_bound} gives
\begin{equation}
    1-\beta_t(\xi)\lesssim_\alpha t^{-\kappa}|\xi|^{2\kappa}, \qquad 0\leq\kappa\leq1.
\end{equation}
Taking $\kappa=1-\frac{\alpha}{2}$ gives
\[
|\xi|^{\alpha-2}\left(1-\beta_t(\xi)\right)
\lesssim t^{-1+\frac{\alpha}{2}},
\]
and hence, by Plancherel’s theorem,
\begin{align*}
    \|\nabla_y\varphi_\hi(t)\|_{H^N}&\lesssim \left(\int_{\R^3}\langle\xi\rangle^{2N} |\xi|^{2\alpha-4}\left(1-\beta_t(\xi) \right)^2 |\widehat m(t,\xi)|^2\,d\xi
    \right)^{1/2}\\
    &\lesssim t^{-1+\frac{\alpha}{2}}\left(\int_{\R^3}\langle\xi\rangle^{2N}|\widehat m(t,\xi)|^2\,d\xi\right)^{1/2}=t^{-1+\frac{\alpha}{2}}\|m(t)\|_{H^N}.
\end{align*}
\end{proof}

We now use these estimates to bound the actions. 
\begin{lemma}\label{lem:action_est}
    Under the bootstrap assumptions \eqref{boot}, for $1\leq t\leq T$, we have
\begin{equation}\label{eq:est_Gamma}
    \|\Gamma(t)\|_{L^\infty}+\|\nabla_y \Gamma(t)\|_{H^{N+1-\alpha}}\lesssim\int_1^t s^{-\alpha} \left(\|m(s)\|_{L^1}+\|m(s)\|_{H^N}\right)\,ds,
\end{equation}
and 
\begin{equation}\label{eq:est_nabla_Gamma}
    \|\nabla_y\Gamma(t)\|_{H^{N+1}}\lesssim \int_1^t s^{-\frac{\alpha}{2}}\left(\|m(s)\|_{L^1}+\|m(s)\|_{H^N}\right)\,ds.
\end{equation}
\end{lemma}
\begin{proof}
    Differentiating the action equation \eqref{eq:theta} gives
    \begin{equation}\label{eq:PDE_Gamma}
        \partial_t\Gamma-t^{-2}(\Gamma\cdot\nabla_y) \Gamma=t^{-\alpha}\nabla_y\varphi_\lo,
    \end{equation}
    and 
    \begin{equation} \label{eq:PDE_nabla_Gamma}
        \partial_t \nabla_y\Gamma-t^{-2} \Gamma\cdot\nabla_y^2 \Gamma =t^{-2}(\nabla_y\Gamma)^2+t^{-\alpha} \nabla_y^2\varphi_\lo,
    \end{equation}
    where $(\nabla_y\Gamma)^2$ denotes the matrix product. Since $\Gamma(1)=0$, integration along the characteristics of $-t^{-2}\Gamma$ yields
    \begin{align*}
        \|\Gamma(t)\|_{L^\infty}\leq \int_1^t s^{-\alpha}\|\nabla_y\varphi_\lo(s)\|_{L^\infty}\,ds,
    \end{align*}
    which proves the $L^{\infty}$ part \eqref{eq:est_Gamma} by \eqref{eq:nabla_varphi}. 

    We next estimate $\nabla_y\Gamma$ in $H^{N+1-\alpha}$ and $H^{N+1}$. The commutator estimate in Lemma \ref{lem:KP} and Sobolev embedding give
    \begin{align*}
        \|[\langle D \rangle^r, \Gamma\cdot\nabla_y]\nabla_y \Gamma\|_{L^2}\lesssim \|\nabla_y\Gamma\|_{L^\infty}\|\nabla_y\Gamma\|_{H^r} +\|\nabla_y\Gamma\|_{H^{r-1}} \|\nabla_y^2 \Gamma\|_{L^\infty} \lesssim\|\nabla_y\Gamma\|_{H^3}\|\nabla_y\Gamma\|_{H^r},
    \end{align*}
    for both $r=N+1-\alpha$ and $r=N+1$, since either case has $r\geq 3$. Together with the product estimate for $(\nabla_y \Gamma)^2$ and integration by parts in the transport term, it gives
    \begin{align*}
        \frac{d}{dt}\|\nabla_y \Gamma(t)\|_{H^r} \lesssim t^{-2}\|\nabla_y\Gamma(t)\|_{H^3}\|\nabla_y \Gamma(t)\|_{H^r}+t^{-\alpha}\|\nabla_y^2\varphi_\lo(t)\|_{H^r}.
    \end{align*}
    The bootstrap assumption \eqref{boot} implies
    \[
    \int_1^T t^{-2}\|\nabla_y\Gamma(t)\|_{H^3}\,dt \leq\varepsilon^2\int_1^T t^{-1-\alpha}\,dt \leq\frac{\varepsilon^2}{\alpha}.
    \]
    Since $\nabla_y\Gamma(1)=0$, Gronwall’s inequality therefore yields
    \begin{align*}
        \|\nabla_y\Gamma(t)\|_{H^r}\lesssim \int_1^t s^{-\alpha}\|\nabla_y^2 \varphi_\lo(s)\|_{H^r}\,ds.
    \end{align*}
    Lemma \ref{lem:est_varphi} then proves the claimed estimates \eqref{eq:est_Gamma} and \eqref{eq:est_nabla_Gamma}. 
\end{proof}


\subsection{The Hartree equation} \label{subsec:boot_hartree}
We now consider the Hartree equation. For sufficiently small $\varepsilon_0$, local $H^1$ theory (see \cite{GVGWP}) gives a solution on $[0,1]$. We first explain how the spatial weights in the initial-data assumption \eqref{eq:initial-Hartree} provide the required profile regularity at $t=1$. 

For $0<t\leq1$, we define $g$ by the same transformation as in \eqref{eq:g}.  Recall  Lemma \ref{lem:galilean_profile} gives the identity
\begin{equation}\label{eq:Galilean_g}
    \partial_y^\mu g(t,y) = (-i)^{|\mu|}t^{\frac{3}{2}}e^{-it\frac{|y|^2}{2}}(G(t)^\mu u(t))(ty),
\end{equation}
and 
\begin{equation}
    \|\partial_y^\mu g(t)\|_{L_y^2} =\|G(t)^\mu u(t)\|_{L_x^2}.
\end{equation}
Thus \emph{spatial derivatives} of the ray profile correspond to \emph{Galilean derivatives} of the solution. Since $G(0)=x$, their initial values are precisely the spatial moments $x^\mu u_0$. In particular, we have
\begin{align*}
    \limsup_{t\to0^+}\|g(t)\|_{H^N} \lesssim\|\langle x\rangle^N u_0\|_{L^2}\lesssim\varepsilon_0.
\end{align*}
Recall that  the profile $g$ satisfies
\[
i\partial_tg+\frac{1}{2t^2}\Delta_yg=t^{-\alpha}\varphi g,\qquad\varphi= V_\alpha*|g|^2.
\]
A standard energy estimate with integration by parts and the reality of the potential then gives
\[
\partial_t \|g(t)\|_{H^N}\lesssim t^{-\alpha}
\|\nabla_y\varphi(t)\|_{W^{N-1,\infty}}\|g(t)\|_{H^N}\lesssim t^{-\alpha}\|g(t)\|_{H^N}^3,
\]
using the potential estimate in Lemma \ref{lem:est_varphi}. Hence, a standard continuity argument gives
\begin{equation}\label{eq:small_time_g}
    \sup_{0< t\leq 1}\|g(t)\|_{H^N}\lesssim\varepsilon_0.
\end{equation}
Since $\theta(1)=0$, we have $h(1)=g(1)$, and hence the same bound holds for the corrected profile. Recall from Subsection \ref{subsec:correction} that $m=|h|^2$ and that the corrected Hartree profile satisfies
\[
i\partial_t h +\frac{1}{2t^2}\Delta_y h-\frac{i}{t^2}\mathcal T_\theta h=t^{-\alpha}\varphi_{\hi}h, \qquad \mathcal{T}_\theta=\Gamma\cdot \nabla_y+\frac{1}{2}\nabla_y\cdot \Gamma.
\]
We now close the bootstrap \eqref{boot} for Hartree.
\begin{lemma}\label{lem:close_hartree}
    Under the bootstrap assumptions \eqref{boot} and initial-data condition \eqref{eq:initial-Hartree}, we have
    \begin{equation}\label{eq:improved_h}
        \sup_{1\leq t\leq T}\|h(t)\|_{H^N}
        \lesssim\varepsilon_0.
    \end{equation}
    Hence,
    \begin{equation}\label{eq:improved_m}
        \sup_{1\leq t\leq T}\|m(t)\|_{H^N} \lesssim\varepsilon_0^2.
    \end{equation}
    Moreover,
    \begin{equation}\label{eq:improved_Gamma}
        \|\Gamma(t)\|_{L^\infty}+ \|\nabla_y\Gamma(t)\|_{H^{N+1-\alpha}}\lesssim\varepsilon_0^2t^{1-\alpha},
    \end{equation}
    and 
    \begin{equation}\label{eq:improved_regularity_Gamma}
        \|\nabla_y\Gamma(t)\|_{H^{N+1}} \lesssim\varepsilon_0^2t^{1-\frac{\alpha}{2}}.
    \end{equation}
For sufficiently small $\varepsilon_0$, these estimates close the bootstrap \eqref{boot}. The residual current \eqref{eq:current} therefore satisfies
    \begin{equation}\label{eq:current_est}
        \sup_{t\geq1}\left(\|\mathcal{J}(t)\|_{W^{N-1,1}}+\|\mathcal J(t)\|_{H^{N-1}}\right)\lesssim \varepsilon_0^2.
    \end{equation}
\end{lemma}
\begin{proof}
    Since $\mathcal T_\theta$ is skew-adjoint and $\varphi_{\mathrm{hi}}$ is real-valued, they do not contribute to the evolution of $h$; hence
    \begin{equation}\label{eq:conserved_L2_h}
        \|h(t)\|_{L^2}=\|h(1)\|_{L^2}\lesssim\varepsilon_0.
    \end{equation}
    For $|\mu|\leq N$, the differentiated energy identity is
    \begin{align*}
        \frac{d}{dt}\|\partial_y^\mu h\|_{L^2}^2= 2t^{-2}\Re \langle[\partial_y^\mu,\mathcal T_\theta]h, \partial_y^\mu h\rangle +2t^{-\alpha}\Im \langle[\partial_y^\mu,\varphi_\hi]h,\partial_y^\mu h\rangle.
    \end{align*}
    We first decompose the transport commutator as
    \begin{align*}
        [\partial_y^\mu,\mathcal T_\theta]h=\sum_{j=1}^3[\partial_y^\mu,\Gamma_j]\partial_{y_j}h+\frac12 [\partial_y^\mu,\nabla_y\cdot\Gamma]h.
    \end{align*}
    Applying the commutator estimate in Lemma \ref{lem:commutator} to each term gives 
    \begin{equation}\label{eq:transport_commutator_est}
        \begin{split}
            \|[\partial_y^\mu,\mathcal T_\theta] h\|_{L^2}&\lesssim \|\nabla_y\Gamma\|_{H^{N-1}}\|\nabla_yh\|_{H^{N-1}}+\|\nabla_y(\nabla_y\cdot\Gamma)\|_{H^{N-1}} \|h\|_{H^{N-1}}\\
        &\lesssim \|\nabla_y\Gamma\|_{H^N}\|h\|_{H^N}.
        \end{split}
    \end{equation}
    Here, the additional derivative of $\Gamma$ comes from the divergence term in $\mathcal T_\theta$. Similarly, the potential commutator is bounded by
    \begin{align*}
        \|[\partial_y^\mu,\varphi_\hi]h\|_{L^2}\lesssim\|\nabla_y\varphi_\hi\|_{H^{N-1}}\|h\|_{H^{N-1}}\lesssim\|\varphi_\hi\|_{H^N}\|h\|_{H^N}.
    \end{align*}
    Using Lemma \ref{lem:est_varphi} and bootstrap assumption \eqref{boot}, we obtain
    \begin{align*}
        \partial_t \|h(t)\|_{H^N} \lesssim\left(t^{-2}\|\nabla_y\Gamma(t)\|_{H^N}+t^{-\alpha}\|\varphi_\hi(t)\|_{H^N}\right)\|h(t)\|_{H^N}\lesssim\varepsilon^2t^{-1-\alpha}\|h(t)\|_{H^N}.
    \end{align*}
    Gr\"onwall’s inequality gives
    \begin{align*}
        \|h(t)\|_{H^N}\lesssim\|h(1)\|_{H^N}\exp\left(C\varepsilon^2\int_1^t s^{-1-\alpha}\,ds\right) \lesssim\varepsilon_0,
    \end{align*}
    which proves \eqref{eq:improved_h}. Since $H^N(\mathbb R^3)$ is an algebra, we have
    \begin{align*}
        \|m(t)\|_{H^N}\lesssim\|h(t)\|_{H^N}^2\lesssim\varepsilon_0^2.
    \end{align*}
    Combining this improved density bound with conservation of mass \eqref{eq:mass_conservation}, we obtain
    \begin{align}\label{eq:density_est}
        \|m(s)\|_{L^1}+\|m(s)\|_{H^N}\lesssim\varepsilon_0^2, \qquad 1\leq s\leq T.
    \end{align}
    Inserting this into Lemma \ref{lem:action_est} gives
    \begin{align*}
        \|\Gamma(t)\|_{L^\infty}+\|\nabla_y\Gamma(t)\|_{H^{N+1-\alpha}}
        &\lesssim\varepsilon_0^2\int_1^t s^{-\alpha}\,ds \lesssim\varepsilon_0^2t^{1-\alpha},\\
        \|\nabla_y\Gamma(t)\|_{H^{N+1}}
        &\lesssim\varepsilon_0^2\int_1^t s^{-\frac{\alpha}{2}} \,ds\lesssim\varepsilon_0^2t^{1-\frac{\alpha}{2}}.
    \end{align*}
    Choosing $\varepsilon_0$ sufficiently small with respect to $\varepsilon$ to absorb the implicit constants, the estimates here improve both bootstrap bounds. Therefore, a standard continuity argument extends these estimates to solutions for all $t\geq 1$. 

Finally, for the current, \eqref{eq:current} and the algebra property of $H^{N-1}$ give 
    \begin{equation}
       \|\mathcal{J}\|_{H^{N-1}}\les \| h\nabla h\|_{H^{N-1}}\les \|h\|_{H^N}^2\les \varepsilon_0^2.
    \end{equation}
 Moreover,
 \begin{equation*}
     \begin{split}
         \|\mathcal{J}\|_{W^{N-1,1}}\les \sum_{|\mu|\leq N-1}\sum_{\nu\leq \mu}\|\pr_y^\nu h\|_{L^2}\|\pr_y^{\mu-\nu}\nabla_y h\|_{L^2}\les \varepsilon_0^2.
     \end{split}
 \end{equation*}
\end{proof}

\medskip

We next construct the limiting profile. A difference energy estimate controls the remaining free evolution at the cost of one derivative. Interpolation then gives convergence at every Sobolev order below $N$. 
\begin{proposition}\label{prop:hartree_h_convergence}
    There exists $h_\infty\in H^N$, with $\|h_\infty\|_{H^N}\lesssim\varepsilon_0$, such that 
    \begin{equation}\label{eq:hartree_h_convergence}
        \|h(t)-h_\infty\|_{H^{N-1}} \lesssim\varepsilon_0t^{-\alpha}.
    \end{equation}
    Hence, for every $s<N$, we also have
    \begin{equation}\label{eq:improved_convergence}
        \|h(t)-h_\infty\|_{H^s}\lesssim_s \varepsilon_0t^{-\alpha\min\{1,N-s\}}.
    \end{equation}
    Define
    \begin{equation}\label{eq:m_infty}
        m_\infty(y)=|h_\infty(y)|^2.
    \end{equation}
    Then 
    \begin{equation}
        \|m(t)-m_\infty\|_{H^{N-1}}+\|m(t)-m_\infty\|_{W^{N-1,1}}\lesssim \varepsilon_0^2t^{-\alpha}.
    \end{equation}
\end{proposition}
\begin{proof}
    By the Leibniz rule, 
    \begin{align*}
        \|\mathcal T_\theta h\|_{H^{N-1}} \lesssim\left(\|\Gamma\|_{L^\infty} +\|\nabla_y\Gamma\|_{H^{N-1}}\right) \|h\|_{H^N}.
    \end{align*}
    Here, the undifferentiated coefficient is estimated in $L^\infty$; all other terms contain at least one derivative of $\Gamma$. Lemma \ref{lem:est_varphi} and Lemma \ref{lem:close_hartree} therefore imply
    \begin{equation}\label{eq:small_term_h}
        t^{-2}\|\mathcal T_\theta h\|_{H^{N-1}}+t^{-\alpha} \|\varphi_\hi h\|_{H^{N-1}} \lesssim\varepsilon_0^3t^{-1-\alpha}.
    \end{equation}
    Fix $t\geq1$ and set $w(\tau)=h(\tau)-h(t)$ for $\tau\geq t$. Then
    \begin{equation}
        i\partial_\tau w +\frac{1}{2\tau^2}\Delta_y w+ \frac{1}{2\tau^2}\Delta_y h(t)=i\tau^{-2}\mathcal T_{\theta(\tau)}h(\tau)+\tau^{-\alpha}\varphi_\hi(\tau)h(\tau).
    \end{equation}
    The first term does not contribute in the $H^{N-1}$ energy evolution. For the second term, we integrate one derivative by parts and use the uniform $H^N$ bound:
    \begin{align*}
        \big| \langle\Delta_y h(t),w(\tau)\rangle_{H^{N-1}}\big|&= \big|\langle\nabla_y h(t),\nabla_y w(\tau)\rangle_{H^{N-1}} \big|\lesssim \|h(t)\|_{H^N}\|w(\tau)\|_{H^N} \lesssim\varepsilon_0^2.
    \end{align*}
    Consequently, using \eqref{eq:small_term_h}
    \begin{align*}
        \partial_\tau \|w(\tau)\|_{H^{N-1}}^2 \lesssim \varepsilon_0^2\tau^{-2}+\varepsilon_0^3\tau^{-1-\alpha} \|w(\tau)\|_{H^{N-1}}.
    \end{align*}
    Since $w(t)=0$, integration and taking the supremum over $t\leq\tau\leq S$ yield
    \begin{align*}
        \sup_{t\leq\tau\leq S}\|w(\tau)\|_{H^{N-1}}^2\lesssim\varepsilon_0^2t^{-1}+\varepsilon_0^3t^{-\alpha} \sup_{t\leq\tau\leq S}\|w(\tau)\|_ {H^{N-1}}.
    \end{align*}
    Solving this quadratic inequality, uniformly in $S$, we have
    \begin{equation*}
        \sup_{\tau\geq t}\|h(\tau)-h(t)\|_{H^{N-1}} \lesssim \varepsilon_0t^{-1/2}+\varepsilon_0^3t^{-\alpha}.
    \end{equation*}
    Hence   $h(t)$ is Cauchy in $H^{N-1}$ as $t\rightarrow \infty$. Its limit belongs to $H^N$ by the uniform $H^N$ bound and weak compactness. Letting $\tau\rightarrow\infty$  proves the first assertion, using $\alpha\leq1/2$.

    For $N-1<s<N$, Sobolev interpolation gives
    \begin{align*}
        \|h(t)-h_\infty\|_{H^s} \lesssim \|h(t)-h_\infty\|_{H^{N-1}}^{N-s} \|h(t)-h_\infty\|_{H^N}^{s-N+1}\lesssim_s \varepsilon_0t^{-\alpha(N-s)}.
    \end{align*}
    For $s\leq N-1$, \eqref{eq:improved_convergence} follows from \eqref{eq:hartree_h_convergence} since $\|\cdot \|_{H^s}\leq \|\cdot \|_{H^{N-1}}$ .

    Finally, the fact that $H^{N-1}$ is an algebra since $N-1>\frac{3}{2}$ and the Leibniz rule with Cauchy–Schwarz give
    \begin{align*}
        \|m(t)-m_\infty&\|_{H^{N-1}} +\|m(t)-m_\infty\|_{W^{N-1,1}} \\
        & \lesssim \big(\|h(t)\|_{H^{N-1}}+\|h_\infty\|_{H^{N-1}}\big) \|h(t)-h_\infty\|_{H^{N-1}}\lesssim\varepsilon_0^2t^{-\alpha}.
    \end{align*}
\end{proof}


\subsection{The Vlasov--Riesz equation}\label{subsec:boot_VR}
We now close the bootstrap \eqref{boot} for the Vlasov--Riesz equation and construct the limiting profile. We continue to suppress the kinetic subscript. As in the Hartree argument, we first control the corrected profile, then improve the density and action bounds. The additional difficulty is the momentum factor in the differentiated kinetic equation. We handle it by localizing the initial momentum and transporting each piece with the full Hamiltonian flow. We first rewrite \eqref{eq:PDE_h_K} in the Hamiltonian form:
\begin{equation}\label{eq:PDE_h_Hamiltonian}
    \begin{split}
        \pr_t h+\{h,\mathcal{K}\}=0, \quad  \K=t^{-2}\left[\frac{|\eta|^2}{2}-\Gamma(t,y)\cdot \eta\right]+t^{-\alpha}\varphi_{\hi}(t,y),
    \end{split}
\end{equation}
We also use the uncorrected equation \eqref{eq:g_K}, written as
\begin{equation}\label{eq:g-Hamiltonian}
    \begin{split}
     \pr_t g+\{g,\H\}=0, \quad \H=\frac{|\eta|^2}{2t^2}+t^{-\alpha}\varphi(t,y).
    \end{split}
\end{equation}
Here, the Poisson bracket is as in \eqref{eq:poisson}. Both Hamiltonian flows preserve phase-space volume. {For sufficiently small $\varepsilon_0$, standard weighted local theory (see \cite{HP26}) under the initial data assumption \eqref{eq:initial-VR} gives a solution on $[0,1]$ satisfying
\begin{equation}\label{eq:weighted_small_time}
    \begin{split}
        \sup_{0\leq t\leq 1} \Bigg(&\sum_{|\mu|+|\nu|\leq N-1} \big\|\la x-tv\ra ^{2N+10}\, \pr_x^{\mu}\pr^{\nu}_v f(t)\big\|_{L^\infty_{x,v}}\\
        &\quad\,\, +\sum_{|\mu|+|\nu|\leq N} \big\|\la x-tv\ra ^{2N+10}\, \pr_x^{\mu}\pr^{\nu}_v f(t)\big\|_{L^2_{x,v}}\Bigg)\lesssim\varepsilon_0.
    \end{split}
\end{equation}
The weight is adapted to free transport, and $(\partial_t+v\cdot\nabla_x+F\cdot \nabla_v)(x-tv)=-tF$, so the usual differentiated weighted estimates apply on this bounded interval. At $t=1$, we have
\[
g(1,y,\eta)=f(1,y,y+\eta).
\]
This volume-preserving change of variables sends $v-x$ to $\eta$, while $\partial_y$ and $\partial_\eta$ correspond to $\partial_x+\partial_v$ and $\partial_v$, respectively. Since $\theta(1)=0$, we have $h(1)=g(1)$ and therefore
\begin{equation}\label{eq:kinetic_initial_replacement}
    \sum_{|\mu|+|\nu|\leq N-1}\big\|\la \eta \ra^{2N+10}\partial_y^\mu\partial_\eta^\nu\,h(1)\big\|_{L^\infty}+\sum_{|\mu|+|\nu|\leq N} \big\|\la \eta \ra^{2N+10}\partial_y^\mu\partial_\eta^\nu\,h(1)\big\|_{L^2}\lesssim\varepsilon_0.
\end{equation}
}
\begin{lemma}\label{lem:propagation}
Under the bootstrap assumption \eqref{boot} and the initial-data assumption \eqref{eq:initial-VR}, there are constants $C,c>0$, depending only on $\alpha,N$ and
$\lambda$, such that, if $\varepsilon$ is sufficiently small and
$\varepsilon_0\leq c\varepsilon$, then
\begin{equation}\label{eq:kinetic_propagation}
 \begin{split}
 &\sum_{|\mu|+|\nu|\leq N-1}
 \big\|\la\eta\ra^4
       \partial_y^\mu\partial_\eta^\nu h(t)\big\|_{L^\infty}+\sum_{|\mu|+|\nu|\leq N}\big\|\la \eta\ra\partial_y^\mu\partial_\eta^\nu h(t)\big\|_{L^2}\leq C\varepsilon_0, \qquad 1\leq t\leq T.
 \end{split}
\end{equation}
In particular,
\begin{equation}\label{eq:kinetic_density_improvement}
 \|m(t)\|_{H^N}\leq C\varepsilon_0^2,
 \qquad 1\leq t\leq T.
\end{equation}
Moreover,
\begin{equation}\label{eq:kinetic_currents}
    \begin{split}
        \|\mathcal{J}(t)\|_{W^{N,1}\cap H^N}\les \varepsilon_0^2.
    \end{split}
\end{equation}

\end{lemma}

\begin{proof}
The differentiation rule for Poisson brackets gives
\begin{equation}\label{eq:bracket_commutation}
\begin{split}
 &\partial_t\big(\partial_y^\mu\partial_\eta^\nu h\big)
   +\big\{\partial_y^\mu\partial_\eta^\nu h,\K\big\}=-\sum_{\substack{\mu'\leq\mu,\ \nu'\leq\nu\\
 |\mu'|+|\nu'|>0}} \binom\mu{\mu'}\binom\nu{\nu'}\big\{\partial_y^{\mu-\mu'}\partial_\eta^{\nu-\nu'}h, \partial_y^{\mu'}\partial_\eta^{\nu'}\K\big\}.
\end{split}
\end{equation}
The identical formula holds with $h,\mathcal{K}$ replaced by $g,\H$. To prove \eqref{eq:kinetic_propagation}, we make the additional bootstrap assumption that its left-hand side is at most $\varepsilon$. By \eqref{eq:kinetic_initial_replacement}, this holds initially when $\varepsilon_0\leq c\varepsilon$ and $c$ is sufficiently small. Lemma \ref{lem:action_est}, \eqref{boot}, and conservation of mass \eqref{eq:mass_conservation} give
\begin{align*}
    \|\Gamma(t)\|_{L^\infty}+\|\nabla_y \Gamma(t)\|_{H^{N+1-\alpha}}&\lesssim\varepsilon^2 t^{1-\alpha}, \\
    \|\nabla_y\Gamma(t)\|_{H^{N+1}}&\lesssim \varepsilon^2 t^{1-\frac{\alpha}{2}}. 
\end{align*}
The second bound supplies the coefficient regularity needed at the highest spatial derivatives. Hence, for $|\mu|\leq N-1$
\begin{equation}
    \begin{split}
        |\pr_y^{\mu}m(t,y) |\leq \sum_{\nu\leq \mu}\binom{\mu}{\nu}\int |\pr_y^{\nu}h| |\pr_y^{\mu-\nu} h|d\eta \les \varepsilon^2\int \langle \eta\ra^{-8}d\eta \les \varepsilon^2.
    \end{split}
\end{equation}
Thus $\|m(t)\|_{W^{N-1,\infty}}\les \varepsilon^2$. For the complementary force, let $P_\ell$, $\ell\in\Z$ denote the homogeneous dyadic Fourier projections to $|\xi|\sim2^\ell$. Since
\[
\widehat{\nabla_yV_\alpha}(\xi)=c_\alpha\lambda\,i\xi\,|\xi|^{\alpha-3},
\qquad 1-\beta_t(\xi)\lesssim
\min\{1,t^{-1}|\xi|^2\},
\]
rescaling to a fixed annulus gives
\[
\big\|P_\ell(I-J(t))\nabla_y V_\alpha\big\|_{L^1} \lesssim 2^{(\alpha-2)\ell}\min\{1,t^{-1}2^{2\ell}\}.
\]
Indeed, the rescaled symbol and its derivatives through order four are bounded by the right-hand side, so the kernel estimate follows by integration by parts in Fourier inversion. Summing below and above the frequency $t^{1/2}$, we obtain
\begin{align*}
    \big\|(I-J(t))\nabla_yV_\alpha\big\|_{L^1}\lesssim t^{-1}\sum_{2^\ell\leq t^{1/2}}2^{\alpha\ell} +\sum_{2^\ell>t^{1/2}}2^{(\alpha-2)\ell}\lesssim t^{-1+\frac{\alpha}{2}},
\end{align*}
where both geometric series converge since $0<\alpha<2$. Young's inequality therefore gives
\begin{equation}\label{eq:phi-high-oo}
    \begin{split}
        \|\nabla_y \varphi_{\hi}\|_{W^{N-1,\infty}}\leq \| (I-J(t))\nabla_y V_{\alpha}\|_{L^1}\|m\|_{W^{N-1,\infty}}\les \varepsilon^2 t^{-1+\frac{\alpha}{2}}.
    \end{split}
\end{equation}
To control the momentum factors produced by differentiating $\mathcal K$, we localize the initial data in $\eta$, then transport each piece with the same Hamiltonian. Choose a smooth partition of unity $\sum_{j\geq 0}\chi_j(\eta)=1$\footnote{Here $j\geq 0$ is a dyadic index; the component labels remain suppressed.} with $\chi_0$ supported in $|\eta|\leq 2$ and $\chi_j$ supported in $2^{j-1}\leq |\eta|\leq 2^{j+1}$ for $j \geq 1$, such that 
\[
|\partial_\eta^\mu \chi_j(\eta)|\lesssim_\mu 2^{-j|\mu|}.
\]
Define
\begin{equation*}
    \begin{split}
        \pr_t h_j+\{h_j,\K\}=0,\quad h_j(1,y,\eta)=\chi_j(\eta)h(1,y,\eta).
    \end{split}
\end{equation*}
Linearity with these fixed coefficients gives $h=\sum_{j\geq 0} h_j$. Note that all pieces use the fields generated by the full density $m$. Undoing the momentum translation, we also have
\begin{equation*}
    \begin{split}
        \pr_t g_j+\{g_j,\H\}=0, \quad g_j(t,y,\eta)=h_j(t,y,\eta+\Gamma(t,y)), \quad g_j(1)=h_j(1).
    \end{split}
\end{equation*}
Along any characteristic of $\K$,
\[
h_j(t,y(t),\eta(t))=\chi_j(\eta(1))h(1,y(1),\eta(1)).
\]
The initial derivative identity is
\[
 \partial_y^\mu\partial_\eta^\nu h_j(1) =\sum_{\nu'\leq\nu}\binom{\nu}{\nu'}
   (\partial_\eta^{\nu'}\chi_j)
   \partial_y^\mu\partial_\eta^{\nu-\nu'}h(1).
\]
Thus \eqref{eq:kinetic_initial_replacement} gives, for every
$j\geq0$,
\begin{equation}\label{eq:localized_initial}
\begin{split}
 \sum_{|\mu|+|\nu|\leq N-1}
   \|\partial_y^\mu\partial_\eta^\nu h_j(1)\|_{L^\infty}+\sum_{|\mu|+|\nu|\leq N}\|\partial_y^\mu\partial_\eta^\nu h_j(1)\|_{L^2}\leq C\varepsilon_0\,2^{-j(2N+10)}.
\end{split}
\end{equation}
By \eqref{eq:est_Gamma}, \eqref{boot} and \eqref{eq:phi-high-oo},
\begin{equation}\label{eq:localized_support_ode}
 \left|\frac{d\eta}{dt}\right|
 \leq |\nabla_y\K|\les \varepsilon^2t^{-1-\alpha}|\eta|
       +C\varepsilon^2t^{-1-\frac{\alpha}{2}}.
\end{equation}
For a characteristic issuing from $\operatorname{supp}h_j(1)$,
\begin{equation}\label{eq:localized_support}
\begin{split}
 |\eta(t)|&\leq\left(2^{j+1}
   +C\varepsilon^2\int_1^t s^{-1-\frac{\alpha}{2}}\,ds\right)
 \exp\left(C\varepsilon^2\int_1^t s^{-1-\alpha}\,ds\right)\leq C2^j,\quad 1\leq t\leq T.
\end{split}
\end{equation}
Consequently,
\begin{equation}\label{eq:localized_support_both}
\begin{aligned}
 \operatorname{supp}h_j(t) &\subset\{|\eta|\leq C2^j\},&&1\leq t\leq T,\\
 \operatorname{supp}g_j(t)&\subset\{|\eta|\leq C2^j+\|\Gamma(t)\|_{L^\infty}\}
 \subset\{|\eta|\leq C2^j\},&&1\leq t\leq\min\{2^j,T\}.
\end{aligned}
\end{equation}
These inclusions also hold for every derivative of the profiles.

\medskip
\noindent\textbf{Case I}:
Suppose $1\leq t\leq\min\{2^j,T\}$. We use the equation for $g_j$ on this interval. The weights $t^{|\nu|}$ supply a factor $t^{-1}$ whenever differentiation of the force replaces a spatial derivative by a momentum derivative. The free commutators then produce only polynomial growth, which can be absorbed by the initial momentum weights up to $t=2^j$. Writing $e_i$ for the standard coordinate vectors, \eqref{eq:g-Hamiltonian} gives
\begin{equation}\label{eq:early_weighted_equation}
\begin{split}
\partial_t\big(t^{|\nu|}\partial_y^\mu\partial_\eta^\nu g_j\big) +\big\{t^{|\nu|}\partial_y^\mu\partial_\eta^\nu g_j,\mathcal H\big\}&=|\nu|t^{|\nu|-1}\partial_y^\mu\partial_\eta^\nu g_j-t^{|\nu|-2}\sum_{i:\nu_i>0}\nu_i\partial_y^{\mu+e_i}\partial_\eta^{\nu-e_i}g_j\\
& \quad-t^{|\nu|-\alpha}\sum_{0<\mu'\leq\mu}\binom{\mu}{\mu'}\big\{\partial_y^{\mu-\mu'}\partial_\eta^\nu g_j,          \partial_y^{\mu'}\varphi\big\}.
\end{split}
\end{equation}
For $p=\infty$, $|\mu|+|\nu|\leq N-1$, and for $p=2$,
$|\mu|+|\nu|\leq N$, Lemma \ref{lem:hamiltonian_transport} therefore gives
\begin{equation}\label{eq:early_lemma21}
\begin{split}
 t^{|\nu|}
   \|\partial_y^\mu\partial_\eta^\nu g_j(t)\|_{L^p}
 &\leq\|\partial_y^\mu\partial_\eta^\nu h_j(1)\|_{L^p} +|\nu|\int_1^t s^{|\nu|-1}\|\partial_y^\mu\partial_\eta^\nu g_j(s)\|_{L^p}\,ds\\
& \quad+\sum_{i:\nu_i>0}\nu_i      \int_1^t s^{|\nu|-2} \|\partial_y^{\mu+e_i}\partial_\eta^{\nu-e_i}g_j(s)\|_{L^p}\,ds \\
& \quad+\sum_{0<\mu'\leq\mu}\binom{\mu}{\mu'}\int_1^t s^{|\nu|-\alpha} \big\|\big\{\partial_y^{\mu-\mu'}\partial_\eta^\nu g_j(s),         \partial_y^{\mu'}\varphi(s)       \big\}\big\|_{L^p}\,ds.
\end{split}
\end{equation}
For $1\leq|\mu'|\leq N-1$, using \eqref{eq:nabla_varphi}, the bracket source satisfies 
\begin{equation}\label{eq:early_potential_source}
\begin{split}
 s^{|\nu|-\alpha}\big\|\big\{\partial_y^{\mu-\mu'}\partial_\eta^\nu g_j,\partial_y^{\mu'}\varphi\big\}\big\|_{L^p} &\leq C s^{|\nu|-\alpha}\|\nabla_y\partial_y^{\mu'}\varphi\|_{L^\infty_y}   \sum_{i=1}^3\|\partial_y^{\mu-\mu'}\partial_\eta^{\nu+e_i}g_j\|_{L^p}\\
 &\leq C\varepsilon^2s^{-1-\alpha}   \sum_{i=1}^3 s^{|\nu|+1}\|\partial_y^{\mu-\mu'}\partial_\eta^{\nu+e_i}g_j\|_{L^p}.
\end{split}
\end{equation}
Here $|\mu-\mu'|+|\nu+e_i|\leq|\mu|+|\nu|$. 
The only remaining case is $p=2$, $|\mu|=N$, $\nu=0$,
$\mu'=\mu$. Here \eqref{eq:localized_support_both} and \eqref{eq:nabla2_varphi} give
\begin{equation}\label{eq:early_top_source}
\begin{split}
 s^{-\alpha}\|\{g_j,\partial_y^\mu\varphi\}\|_{L^2}&\leq C s^{-\alpha}\|\nabla_y\partial_y^\mu\varphi\|_{L^2_y}\left(\int_{|\eta|\leq C2^j}d\eta\right)^{1/2}\sum_{i=1}^3\|\partial_{\eta_i}g_j\|_{L^\infty}\\
 &\leq C\varepsilon^2s^{-1-\alpha}2^{3j/2}      \sum_{i=1}^3 s\|\partial_{\eta_i}g_j\|_{L^\infty}.
\end{split}
\end{equation}
For either $n=N-1$ or $n=N$, reindexing the free term yields
\begin{equation}\label{eq:early_free_sum}
\begin{split}
 \sum_{|\mu|+|\nu|\leq n}\sum_{i:\nu_i>0}   \nu_i s^{|\nu|-2}\|\partial_y^{\mu+e_i}\partial_\eta^{\nu-e_i}g_j\|_{L^p}&=s^{-1}\sum_{|\mu|+|\nu|\leq n}\left(\sum_{i:\mu_i>0}(\nu_i+1)\right) s^{|\nu|}\|\partial_y^\mu\partial_\eta^\nu g_j\|_{L^p}\\
 &\leq \frac{N+3}{s}\sum_{|\mu|+|\nu|\leq n}s^{|\nu|}\|\partial_y^\mu\partial_\eta^\nu g_j\|_{L^p},
\end{split}
\end{equation}
whereas
\begin{equation}\label{eq:early_weight_sum}
 \sum_{|\mu|+|\nu|\leq n}|\nu|s^{|\nu|-1}  \|\partial_y^\mu\partial_\eta^\nu g_j\|_{L^p}\leq\frac Ns\sum_{|\mu|+|\nu|\leq n}s^{|\nu|}\|\partial_y^\mu\partial_\eta^\nu g_j\|_{L^p}.
\end{equation}
Multiply the $L^\infty$ sum by $2^{3j/2}$ and add the $L^2$ sum. This factor accounts for the momentum volume in \eqref{eq:early_top_source}, so the highest differentiated field does not require an additional derivative of the profile. Equations \eqref{eq:early_lemma21}--\eqref{eq:early_weight_sum}
give the explicit integral inequality
\begin{equation}\label{eq:early_integral_inequality}
\begin{aligned}
 &2^{3j/2}\sum_{|\mu|+|\nu|\leq N-1}
t^{|\nu|}\|\partial_y^\mu\partial_\eta^\nu g_j(t)\|_{L^\infty}+\sum_{|\mu|+|\nu|\leq N} t^{|\nu|}\|\partial_y^\mu\partial_\eta^\nu g_j(t)\|_{L^2}\\     &\leq 2^{3j/2}\sum_{|\mu|+|\nu|\leq N-1}\|\partial_y^\mu\partial_\eta^\nu h_j(1)\|_{L^\infty}+\sum_{|\mu|+|\nu|\leq N}\|\partial_y^\mu\partial_\eta^\nu h_j(1)\|_{L^2}\\
 &\quad+\int_1^t\left(\frac{2N+3}{s}+C\varepsilon^2s^{-1-\alpha}\right)\\
 &\quad \quad \times\left[   2^{3j/2}\sum_{|\mu|+|\nu|\leq N-1}s^{|\nu|}\|\partial_y^\mu\partial_\eta^\nu g_j(s)\|_{L^\infty}    +\sum_{|\mu|+|\nu|\leq N} s^{|\nu|}\|\partial_y^\mu\partial_\eta^\nu g_j(s)\|_{L^2}  \right]ds.
\end{aligned}
\end{equation}
Since
\begin{equation}\label{eq:early_gronwall_factor}
\begin{split}
 \exp\left(\int_1^t\left(\frac{2N+3}{s}+C\varepsilon^2s^{-1-\alpha}\right)ds\right)= t^{2N+3}\exp\left(\frac{C\varepsilon^2}{\alpha}(1-t^{-\alpha})\right)\leq Ct^{2N+3},
\end{split}
\end{equation}
Gr\"onwall and \eqref{eq:localized_initial} imply
\begin{equation}\label{eq:localized_early}
\begin{split}
 &2^{3j/2}\sum_{|\mu|+|\nu|\leq N-1}
t^{|\nu|}\|\partial_y^\mu\partial_\eta^\nu g_j(t)\|_{L^\infty}+\sum_{|\mu|+|\nu|\leq N}   t^{|\nu|}\|\partial_y^\mu\partial_\eta^\nu g_j(t)\|_{L^2}\leq C\varepsilon_0\,2^{3j/2-j(2N+10)}t^{2N+3}.
\end{split}
\end{equation}
We convert the estimates to $h$. For $p=2,\infty$, the chain rule and invariance of these norms
under momentum translation give
\begin{equation}\label{eq:localized_chain_rule}
\begin{split}
 t^{|\nu|}  \|\partial_y^\mu\partial_\eta^\nu h_j(t)\|_{L^p}\leq C_N\sum_{\substack{0\leq k\leq|\mu|\\
\mu_0+\cdots+\mu_k=\mu\\               |\mu_\ell|\geq1,\ 1\leq\ell\leq k}}\left(\prod_{\ell=1}^k    \frac{\|\partial_y^{\mu_\ell}\Gamma(t)\|_{L^\infty}}{t}
 \right)
 \sum_{|\kappa|=k}t^{|\nu|+k}
       \|\partial_y^{\mu_0}\partial_\eta^{\nu+\kappa}g_j(t)\|_{L^p}.
\end{split}
\end{equation}
For $k=0$, the product equals $1$. By \eqref{eq:est_nabla_Gamma}, \eqref{boot} and Sobolev embedding, each term satisfies  
\[
 |\mu_0|+|\nu|+k\leq|\mu|+|\nu|,
 \qquad
 {\|\partial_y^{\mu_\ell}\Gamma(t)\|_{L^\infty}}
 \leq C\varepsilon^2t^{1-\frac{\alpha}{2}}
 \quad(1\leq|\mu_\ell|\leq N).
\]
Hence, with $n= N-1$ for $p=\infty$ and $n=N$ for $p=2$,
\begin{equation}\label{eq:localized_chain_sum}
 \sum_{|\mu|+|\nu|\leq n}t^{|\nu|}
    \|\partial_y^\mu\partial_\eta^\nu h_j(t)\|_{L^p} \leq C\sum_{|\mu|+|\nu|\leq n}t^{|\nu|}\|\partial_y^\mu\partial_\eta^\nu g_j(t)\|_{L^p}.
\end{equation}
In particular, \eqref{eq:localized_early} also holds with $g_j$
on the left replaced by $h_j$. If $2^j\leq T$, evaluation at
$t=2^j$ gives
\begin{equation}\label{eq:localized_matching}
\begin{split}
 2^{3j/2}\sum_{|\mu|+|\nu|\leq N-1}2^{j|\nu|}\|\partial_y^\mu\partial_\eta^\nu h_j(2^j)\|_{L^\infty}+\sum_{|\mu|+|\nu|\leq N}2^{j|\nu|}\|\partial_y^\mu\partial_\eta^\nu h_j(2^j)\|_{L^2} \\
 \leq C\varepsilon_0\,2^{3j/2-j(2N+10)+j(2N+3)}  =C\varepsilon_0\,2^{-11j/2}.
\end{split}
\end{equation}

\medskip

\noindent\textbf{Case II}. 
Suppose $2^j\leq t\leq T$. We now use the corrected equation with the fixed weights $2^{j|\nu|}$. The support bound $|\eta|\lesssim 2^j$ allows each momentum factor to be absorbed by the weight on one additional momentum derivative. The free commutator has coefficient $2^j t^{-2}$, whose time integral on this interval is uniformly bounded in $j$. 

Applying Lemma \ref{lem:hamiltonian_transport} to
\eqref{eq:bracket_commutation} for $h_j$ yields
\begin{equation}\label{eq:late_lemma21}
\begin{split}
 2^{j|\nu|}
   \|\partial_y^\mu\partial_\eta^\nu h_j(t)\|_{L^p}&\leq 2^{j|\nu|}\|\partial_y^\mu\partial_\eta^\nu h_j(2^j)\|_{L^p}\\
 &\quad+\sum_{\substack{\mu'\leq\mu,\ \nu'\leq\nu\\                |\mu'|+|\nu'|>0}}\binom{\mu}{\mu'}\binom{\nu}{\nu'}\int_{2^j}^t 2^{j|\nu|}\big\|\big\{
   \partial_y^{\mu-\mu'}\partial_\eta^{\nu-\nu'}h_j(s),\partial_y^{\mu'}\partial_\eta^{\nu'}\K(s)\big\}\big\|_{L^p}\,ds.
\end{split}
\end{equation}
The derivative range is the same as in \eqref{eq:early_lemma21}.
The Hamiltonian derivatives appearing here satisfy
\begin{equation}\label{eq:late_hamiltonian_derivatives}
\begin{aligned}
 \partial_y^{\mu'}\K
   &=-t^{-2}(\partial_y^{\mu'}\Gamma)\cdot\eta
       +t^{-\alpha}\partial_y^{\mu'}\varphi_{\mathrm{high}},    \qquad|\mu'|\geq1,\\
 \partial_y^{\mu'}\partial_{\eta_i}\K  &=t^{-2}\big(\mathbf1_{\{\mu'=0\}}\eta_i-\partial_y^{\mu'}\Gamma_i\big),\qquad|\mu'|\geq0,\\
\partial_y^{\mu'}\partial_{\eta_i}\partial_{\eta_k}\K   &=t^{-2}\mathbf1_{\{\mu'=0\}}\delta_{ik},\\
\partial_y^{\mu'}\partial_\eta^{\nu'}\K&=0,       \qquad|\nu'|\geq3.
\end{aligned}
\end{equation}
Terms with two or more momentum derivatives of $\K$ have zero Poisson bracket, since the corresponding Hamiltonian derivatives are constant or zero. For the free term,
\begin{equation}\label{eq:late_free_source}
 2^{j|\nu|}t^{-2}\big\|\{\partial_y^\mu\partial_\eta^{\nu-e_i}h_j,\eta_i\}\big\|_{L^p} =\frac{2^j}{t^2}\,2^{j(|\nu|-1)}\|\partial_y^{\mu+e_i}\partial_\eta^{\nu-e_i}h_j\|_{L^p}.
\end{equation}
For $1\leq|\mu'|\leq N-1$, the other bracket bounds are
\begin{equation}\label{eq:late_action_source}
\begin{split}
 2^{j|\nu|}\big\|\big\{\partial_y^{\mu-\mu'}\partial_\eta^\nu h_j,   (\partial_y^{\mu'}\Gamma)\cdot\eta\big\}\big\|_{L^p}&\leq C\|\partial_y^{\mu'}\Gamma\|_{L^\infty}\sum_{i=1}^3 2^{j|\nu|} \|\partial_y^{\mu-\mu'+e_i}\partial_\eta^\nu h_j\|_{L^p}\\
&\quad+C\|\nabla_y\partial_y^{\mu'}\Gamma\|_{L^\infty}      \sum_{i=1}^3 2^{j(|\nu|+1)}\|\partial_y^{\mu-\mu'}\partial_\eta^{\nu+e_i}h_j\|_{L^p},
\end{split}
\end{equation}
and
\begin{equation}\label{eq:late_potential_source}
\begin{split}
 2^{j|\nu|}\big\|\big\{\partial_y^{\mu-\mu'}\partial_\eta^\nu h_j,\partial_y^{\mu'}\varphi_{\hi}\big\}\big\|_{L^p}\leq C2^{-j}\|\nabla_y\partial_y^{\mu'}\varphi_\hi\|_{L^\infty}\sum_{i=1}^3 2^{j(|\nu|+1)}\|\partial_y^{\mu-\mu'}\partial_\eta^{\nu+e_i}h_j\|_{L^p}.
\end{split}
\end{equation}
When $\nu'=e_i$, and $0\leq |\mu'|\leq N-1$, the action term satisfies
\begin{equation}\label{eq:late_mixed_source}
\begin{split}
 2^{j|\nu|} \big\|\big\{\partial_y^{\mu-\mu'}\partial_\eta^{\nu-e_i}h_j,\partial_y^{\mu'}\Gamma_i\big\}\big\|_{L^p}\leq C\|\nabla_y\partial_y^{\mu'}\Gamma_i\|_{L^\infty}\sum_{k=1}^3 2^{j|\nu|}\|\partial_y^{\mu-\mu'}\partial_\eta^{\nu-e_i+e_k}h_j\|_{L^p}.
\end{split}
\end{equation}
Only $p=2$, $|\mu|=|\mu'|=N$, $\nu=\nu'=0$ requires a field coefficient in $L_y^2$. For this case,
\begin{equation}\label{eq:late_top_action_source}
\begin{split}
 t^{-2}\|\{h_j,(\partial_y^\mu\Gamma)\cdot\eta\}\|_{L^2}\leq Ct^{-2}\|\partial_y^\mu\Gamma\|_{L^\infty}           \|\nabla_y h_j\|_{L^2}+Ct^{-2}2^{3j/2}\|\nabla_y\partial_y^\mu\Gamma\|_{L^2_y}\big(2^j\|\nabla_\eta h_j\|_{L^\infty}\big),
\end{split}
\end{equation}
and
\begin{equation}\label{eq:late_top_potential_source}
 t^{-\alpha}  \|\{h_j,\partial_y^\mu\varphi_{\hi}\}\|_{L^2} \leq Ct^{-\alpha}2^{3j/2-j}\|\nabla_y\partial_y^\mu\varphi_\hi\|_{L^2_y}   \big(2^j\|\nabla_\eta h_j\|_{L^\infty}\big).
\end{equation}
Using \eqref{eq:phi-high-oo}, \eqref{eq:varphi_hi_derivative} and
\eqref{eq:est_nabla_Gamma}, together with \eqref{boot} and \eqref{eq:mass_conservation}, the summed source bounds are
\begin{equation}\label{eq:late_source_infinity}
\begin{split}
 \sum_{|\mu|+|\nu|\leq N-1}  \sum_{\substack{\mu'\leq\mu,\ \nu'\leq\nu\\
    |\mu'|+|\nu'|>0}}\binom{\mu}{\mu'}&\binom{\nu}{\nu'}2^{j|\nu|}\times\big\|\big\{\partial_y^{\mu-\mu'}\partial_\eta^{\nu-\nu'}h_j,\partial_y^{\mu'}\partial_\eta^{\nu'}\K\big\}\big\|_{L^\infty}\\
 &\leq C\big(2^jt^{-2}+\varepsilon^2t^{-1-\frac{\alpha}{2}}\big)\sum_{|\mu|+|\nu|\leq N-1}2^{j|\nu|}\|\partial_y^\mu\partial_\eta^\nu h_j\|_{L^\infty},
\end{split}
\end{equation}
and
\begin{equation}\label{eq:late_source_two}
\begin{split}
 \sum_{|\mu|+|\nu|\leq N}  \sum_{\substack{\mu'\leq\mu,\ \nu'\leq\nu\\
    |\mu'|+|\nu'|>0}}\binom{\mu}{\mu'}\binom{\nu}{\nu'}2^{j|\nu|}&\times\big\|\big\{\partial_y^{\mu-\mu'}\partial_\eta^{\nu-\nu'}h_j,\partial_y^{\mu'}\partial_\eta^{\nu'}\K\big\}\big\|_{L^2}\\
 &\leq C\big(2^jt^{-2}+\varepsilon^2t^{-1-\frac{\alpha}{2}}\big)\sum_{|\mu|+|\nu|\leq N}2^{j|\nu|}\|\partial_y^\mu\partial_\eta^\nu h_j\|_{L^2}\\
 &\quad\,\,\,\,\,+C\varepsilon^2t^{-1-\frac{\alpha}{2}}2^{3j/2}\sum_{|\mu|+|\nu|\leq N-1}2^{j|\nu|}\|\partial_y^\mu\partial_\eta^\nu h_j\|_{L^\infty}.
\end{split}
\end{equation}
Substitution into \eqref{eq:late_lemma21}, followed by summation,
gives
\begin{equation}\label{eq:late_integral_inequality}
\begin{aligned}
 &2^{3j/2}\sum_{|\mu|+|\nu|\leq N-1}     2^{j|\nu|}\|\partial_y^\mu\partial_\eta^\nu h_j(t)\|_{L^\infty}+\sum_{|\mu|+|\nu|\leq N}    2^{j|\nu|}\|\partial_y^\mu\partial_\eta^\nu h_j(t)\|_{L^2}\\
 &\leq
 2^{3j/2}\sum_{|\mu|+|\nu|\leq N-1}
2^{j|\nu|}\|\partial_y^\mu\partial_\eta^\nu h_j(2^j)\|_{L^\infty}+\sum_{|\mu|+|\nu|\leq N}
2^{j|\nu|}\|\partial_y^\mu\partial_\eta^\nu h_j(2^j)\|_{L^2}\\
&\quad+C\int_{2^j}^t\big(2^js^{-2}+\varepsilon^2s^{-1-\frac{\alpha}{2}}\big)\\
&\quad \quad \times\left[    2^{3j/2}\sum_{|\mu|+|\nu|\leq N-1}       2^{j|\nu|}\|\partial_y^\mu\partial_\eta^\nu h_j(s)\|_{L^\infty} +\sum_{|\mu|+|\nu|\leq N}2^{j|\nu|}\|\partial_y^\mu\partial_\eta^\nu h_j(s)\|_{L^2}    \right]ds.
\end{aligned}
\end{equation}
Here the Gr\"onwall exponent satisfies
\begin{equation}\label{eq:late_gronwall_factor}
\begin{split}
 \int_{2^j}^t   \big(2^js^{-2}+\varepsilon^2s^{-1-\frac{\alpha}{2}}\big)\,ds &=1-\frac{2^j}{t}  +\frac{2\varepsilon^2}{\alpha}      \big(2^{-j\frac{\alpha}{2}}-t^{-\frac{\alpha}{2}}\big)\leq1+\frac{2\varepsilon^2}{\alpha}.
\end{split}
\end{equation}
Therefore \eqref{eq:localized_matching} and Gronwall yield
\begin{equation}\label{eq:localized_late}
\begin{split}
 2^{3j/2}\sum_{|\mu|+|\nu|\leq N-1}2^{j|\nu|}\|\partial_y^\mu\partial_\eta^\nu h_j(t)\|_{L^\infty}+\sum_{|\mu|+|\nu|\leq N}    2^{j|\nu|}\|\partial_y^\mu\partial_\eta^\nu h_j(t)\|_{L^2} \leq C\varepsilon_0\,2^{-11j/2}.
\end{split}
\end{equation}
For $j=0$, the late interval starts at $1$ and
$2^{j|\nu|}=2^{3j/2}=1$.

\medskip
Finally, we sum the estimates obtained from Case I and Case II. For $t\leq2^j$, \eqref{eq:localized_early} and
\eqref{eq:localized_chain_sum} give the same bound as
\eqref{eq:localized_late} after dropping the derivative weights:
\[
 2^{3j/2-j(2N+10)}t^{2N+3}
 \leq2^{3j/2-j(2N+10)+j(2N+3)}
 =2^{-11j/2}.
\]
Consequently, for every $j\geq0$ and $1\leq t\leq T$,
\begin{equation}\label{eq:localized_uniform}
\begin{split}
 &2^{3j/2}\sum_{|\mu|+|\nu|\leq N-1}
     \|\partial_y^\mu\partial_\eta^\nu h_j(t)\|_{L^\infty}+\sum_{|\mu|+|\nu|\leq N}
     \|\partial_y^\mu\partial_\eta^\nu h_j(t)\|_{L^2}
 \leq C\varepsilon_0\,2^{-11j/2}.
\end{split}
\end{equation}
Since $\la\eta\ra^4\leq C2^{4j}$ on the support of every derivative of $h_j$, this implies
\begin{equation}\label{eq:localized_weighted_infinity}
\begin{split}
 \sum_{|\mu|+|\nu|\leq N-1}   \|\la\eta\ra^4     \partial_y^\mu\partial_\eta^\nu h_j(t)\|_{L^\infty}
 &\leq C2^{4j}   \sum_{|\mu|+|\nu|\leq N-1}     \|\partial_y^\mu\partial_\eta^\nu h_j(t)\|_{L^\infty}\\
 &\leq C\varepsilon_0\,2^{4j-3j/2-11j/2}   =C\varepsilon_0\,2^{-3j},
\end{split}
\end{equation}
and
\begin{equation}\label{eq:localized_uniform_two}
 \sum_{|\mu|+|\nu|\leq N}
   \|\partial_y^\mu\partial_\eta^\nu h_j(t)\|_{L^2}
 \leq C\varepsilon_0\,2^{-11j/2}.
\end{equation}
Thus both differentiated series converge in the required norms. The identity $h=\sum_{j\geq0}h_j$ and the triangle inequality give
\begin{equation}\label{eq:localized_final_sum}
\begin{aligned}
 &\sum_{|\mu|+|\nu|\leq N-1}   \|\la\eta\ra^4    \partial_y^\mu\partial_\eta^\nu h(t)\|_{L^\infty}+\sum_{|\mu|+|\nu|\leq N}   \|\la \eta\rangle\partial_y^\mu\partial_\eta^\nu h(t)\|_{L^2}\\
 &\lesssim\sum_{j=0}^{\infty}   \left[   \sum_{|\mu|+|\nu|\leq N-1}     \|\la\eta\ra^4       \partial_y^\mu\partial_\eta^\nu h_j(t)\|_{L^\infty}+ 2^j\sum_{|\mu|+|\nu|\leq N}       \|\partial_y^\mu\partial_\eta^\nu h_j(t)\|_{L^2}  \right]\\
 &\leq C\varepsilon_0 \sum_{j=0}^{\infty}\big(2^{-3j}+2^j\cdot 2^{-11j/2}\big)\leq C\varepsilon_0.
\end{aligned}
\end{equation}
For $\varepsilon_0\leq c\varepsilon$ with $c$ sufficiently
small, the last bound is at most $\varepsilon/2$, closing
the  bootstrap and proving
\eqref{eq:kinetic_propagation}.

Leibniz's rule and Cauchy--Schwarz in $\eta$ also give the density bound. In each product, we put the factor with fewer spatial derivatives in the momentum-weighted $L^\infty$ norm
and the other factor in $L^2$:
\[
 \|m(t)\|_{H^N}\leq C\left( \sum_{|\mu|+|\nu|\leq N-1}      \|\la\eta\ra^4          \partial_y^\mu\partial_\eta^\nu h(t)\|_{L^\infty}\right) \left( \sum_{|\mu|+|\nu|\leq N}      \|\partial_y^\mu\partial_\eta^\nu h(t)\|_{L^2}\right)\les \varepsilon_0^2.
\]
This proves \eqref{eq:kinetic_density_improvement}. 

Finally, the current estimates \eqref{eq:kinetic_currents} follow from Leibniz's rule and Cauchy--Schwarz. For the $H^N$ estimate, we place the factor with fewer derivatives in weighted $L^\infty$ and use $\lfloor N/2\rfloor\le N-2$ and $|\eta|\langle\eta\rangle^{-4}\in L^2_\eta$. Thus,
\begin{equation*}
\begin{split}
\|\mathcal J(t)\|_{H^N}&\lesssim\left(\sum_{|\nu|\le N-2}\|\langle\eta\rangle^4\partial_y^\nu h(t)\|_{L^\infty}\right)\|h(t)\|_{H^N}\lesssim\varepsilon_0^2,\\
\|\mathcal J(t)\|_{W^{N,1}}&\lesssim\sum_{|\mu|\le N}\sum_{\nu\le\mu}\|\langle\eta\rangle\partial_y^\nu h(t)\|_{L^2}\|\partial_y^{\mu-\nu}h(t)\|_{L^2}\lesssim\varepsilon_0^2.
\end{split}
\end{equation*}
This finishes the proof.

\end{proof}

We now insert the improved density estimate \eqref{eq:kinetic_density_improvement}, together with \eqref{eq:mass_conservation}, into Lemma \ref{lem:action_est}. This gives
\begin{align*}
    \|\Gamma(t)\|_{L^\infty}+\|\nabla_y \Gamma(t)\|_{H^{N+1-\alpha}}&\lesssim\varepsilon_0^2 \int_1^t s^{-\alpha}\,ds\,\lesssim \,\varepsilon_0^2 t^{1-\alpha}, \\
    \|\nabla_y\Gamma(t)\|_{H^{N+1}}&\lesssim \varepsilon_0^2 \int_1^t s^{-\frac{\alpha}{2}}ds\,\lesssim\, \varepsilon_0^2 t^{1-\frac{\alpha}{2}}.
\end{align*}
For $\varepsilon_0$ sufficiently small relative to $\varepsilon$, these estimates improve both assumptions in \eqref{boot}. Together with \eqref{eq:kinetic_propagation}, the continuation argument extends the bounds to all $t\geq 1$. We next construct the limiting profile and density.
\begin{proposition}\label{prop:convergence_h2}
Let $N\ge4$. There exists $h_\infty\in H^{N-1}\cap L^\infty$ such that
\begin{equation*}
\begin{split}\|h(t)-h_\infty\|_{H^{N-1}}+\|h(t)-h_\infty\|_{L^\infty}\lesssim\varepsilon_0t^{-\alpha}.
\end{split}
\end{equation*}
Consequently, we define $m_\infty(y):=\int h_\infty(y,\eta)^2\,d\eta$ and have
\begin{equation*}\begin{split}\|m(t)-m_\infty\|_{H^{N-1}}+\|m(t)-m_\infty\|_{W^{N-1,1}}\lesssim\varepsilon_0^2t^{-\alpha}.
\end{split}
\end{equation*}
In particular,
\begin{equation*}
\begin{split}\|m_\infty\|_{H^N}\le\liminf_{t\to\infty}\|m(t)\|_{H^N}\lesssim\varepsilon_0^2.
\end{split}
\end{equation*}
\end{proposition}
\begin{proof}
For convergence, the complementary force is needed at one fewer derivative. We can therefore use the full $t^{-1}$ from \eqref{eq:varphi_hi}. By \eqref{eq:kinetic_density_improvement}, \eqref{eq:mass_conservation}, and Sobolev embedding, we have
\begin{equation*}\begin{split}\|\nabla_y\varphi_{\hi}(t)\|_{L^\infty}+\|\nabla_y\varphi_{\hi}(t)\|_{H^{N-1}}\lesssim\varepsilon^2t^{-1}.\end{split}\end{equation*}
Applying Leibniz's rule to the Poisson bracket in \eqref{eq:PDE_h_Hamiltonian}, and using the Sobolev product estimates in the $y$ variable, we obtain the estimate below. The coefficients depend only on $y$, so the products are estimated in three dimensions with values in $L^2_\eta$. We place the undifferentiated $\Gamma$ in $L^\infty$; derivatives falling on the explicit factor $\eta$ give lower-order terms controlled by \eqref{eq:kinetic_propagation}.
\begin{equation*}
\begin{split}
\|\partial_t& h(t)\|_{H^{N-1}}=\|\{h,\mathcal K\}(t)\|_{H^{N-1}}\\
&\lesssim\left[t^{-2}\left(1+\|\Gamma(t)\|_{L^\infty}+\|\nabla_y\Gamma(t)\|_{H^{N-1}}\right)+t^{-\alpha}\|\nabla_y\varphi_{\hi}(t)\|_{H^{N-1}}\right]\sum_{|\mu|+|\nu|\le N}\|\langle\eta\rangle\partial_y^\mu\partial_\eta^\nu h(t)\|_{L^2}\\&\lesssim\varepsilon_0t^{-2}+\varepsilon_0\varepsilon^2t^{-1-\alpha}.
\end{split}
\end{equation*}
Here we used Lemma \ref{lem:action_est} to bound $\|\Gamma(t)\|_{L^\infty}+\|\nabla_y\Gamma(t)\|_{H^{N-1}}\lesssim\varepsilon^2t^{1-\alpha}$. Similarly, \eqref{eq:kinetic_propagation} gives
\begin{equation*}
\begin{split}\|\partial_t h(t)\|_{L^\infty}&=\|\{h,\mathcal K\}(t)\|_{L^\infty}\\&\lesssim\left[t^{-2}\left(1+\|\Gamma(t)\|_{L^\infty}+\|\nabla_y\Gamma(t)\|_{L^\infty}\right)+t^{-\alpha}\|\nabla_y\varphi_{\hi}(t)\|_{L^\infty}\right]\|\langle\eta\rangle\nabla_{y,\eta}h(t)\|_{L^\infty}\\
&\lesssim\varepsilon_0t^{-2}+\varepsilon_0\varepsilon^2t^{-1-\alpha}.
\end{split}
\end{equation*}
These bounds are integrable in time, so $h(t)$ is Cauchy in $H^{N-1}\cap L^\infty$. Therefore, there exists $h_\infty\in H^{N-1}\cap L^\infty$ such that
\begin{equation*}\begin{split}\|h(t)-h_\infty\|_{H^{N-1}}+\|h(t)-h_\infty\|_{L^\infty}&\lesssim\int_t^\infty\left(\varepsilon_0s^{-2}+\varepsilon_0\varepsilon^2s^{-1-\alpha}\right)\,ds\\&\lesssim\varepsilon_0t^{-\alpha}.
\end{split}
\end{equation*}
Interpolation between $L^2$ and $L^\infty$ then gives $\|h(t)-h_\infty\|_{L^p}\lesssim\varepsilon_0t^{-\alpha}$ for every $2\le p\le\infty$.

We next prove convergence of the density. For $|\mu|\le N-1$, Leibniz's rule gives
\begin{equation*}
\begin{split}\partial_y^\mu(m(t)-m_\infty)=\sum_{\nu\le\mu}\binom{\mu}{\nu}\int\partial_y^\nu(h(t)-h_\infty)\partial_y^{\mu-\nu}(h(t)+h_\infty)\,d\eta.
\end{split}
\end{equation*}
By Cauchy--Schwarz in $\eta$ and the embedding $H_y^2(L_\eta^2)\hookrightarrow L_y^\infty L_\eta^2$, placing the factor with fewer $y$ derivatives in $L_y^\infty L_\eta^2$, we obtain
\begin{equation*}
\begin{split}\|m(t)-m_\infty\|_{H^{N-1}}&\lesssim\|h(t)-h_\infty\|_{H^{N-1}}\left(\|h(t)\|_{H^{N-1}}+\|h_\infty\|_{H^{N-1}}\right)\\&\lesssim\varepsilon_0^2t^{-\alpha}.
\end{split}
\end{equation*}
This uses $\lfloor(N-1)/2\rfloor+2\le N-1$, which holds for $N\ge4$. Likewise, Cauchy--Schwarz in $(y,\eta)$ gives
\begin{equation*}
\begin{split}\|m(t)-m_\infty\|_{W^{N-1,1}}&\lesssim\sum_{|\mu|\le N-1}\sum_{\nu\le\mu}\|\partial_y^\nu(h(t)-h_\infty)\|_{L^2}\|\partial_y^{\mu-\nu}(h(t)+h_\infty)\|_{L^2}\\&\lesssim\varepsilon_0^2t^{-\alpha}.
\end{split}
\end{equation*}
Finally, the uniform $H^N$ bound \eqref{eq:kinetic_density_improvement}, the strong convergence of $m(t)$ in $H^{N-1}$, and weak lower semicontinuity imply
\begin{equation*}
\begin{split}\|m_\infty\|_{H^N}\le\liminf_{t\to\infty}\|m(t)\|_{H^N}\lesssim\varepsilon_0^2.
\end{split}
\end{equation*}
This finishes the proof.
\end{proof}


\section{Action expansions and modified scattering}\label{sec:asymptotics}
In this section, we identify the asymptotic action through its \emph{coupling} with the rescaled density. For either scalar equation, we suppress the component subscript and recall the action–density system:
\begin{equation}\label{eq:action--density}
\begin{aligned}
 \partial_t\theta-\frac1{2t^2}|\nabla\theta|^2 &=t^{-\alpha}\varphi_{\lo},\\
 \partial_tm-\frac1{t^2} \nabla\cdot(m\nabla\theta)&=-\frac1{t^2}\nabla\cdot\mathcal J,
 \qquad \theta(1)=0.
\end{aligned}
\end{equation}
This common action–density structure allows us to derive the asymptotic actions for both equations by the same argument. To index the density coefficients consistently in the recursive expansion, we write
\[
m_0:=m_\infty,
\]
which is the limiting density obtained in Subsections \ref{subsec:boot_hartree} and \ref{subsec:boot_VR}. 

In Subsection \ref{subsec:expansion}, we construct the successive action coefficients from the limiting density and justify the resulting expansion. We then apply this construction to the Hartree and Vlasov–Riesz equations in Subsections \ref{subsec:hartree-scatter} and \ref{subsec:VR-scatter}, respectively, completing the proofs of Theorems \ref{thm:Hartree} and \ref{thm:VR}. 

\subsection{The recursive expansion}\label{subsec:expansion}
Recall that
\[
L=\left\lfloor\frac1\alpha\right\rfloor-1,\qquad
(L+1)\alpha\le1<(L+2)\alpha,\qquad \Lambda_\beta(t)=\int_1^t s^{-\beta}ds.
\]
The recursion follows the action–density feedback described in Subsection \ref{subsec:feedback}. At each order, the known density flux determines the next density coefficient. Its potential then combines with the quadratic action term to determine the next action coefficient. 

\begin{proposition}\label{prop:expansion}
    Let $0<\alpha\leq \frac{1}{2}$, $N\in \N$ with $N\geq L+3$, and $0<\varepsilon\ll 1$. Assume that the real-valued density $m$, action $\theta$, and vector current $\mathcal J$ satisfy the scalar action--density system \eqref{eq:action--density}, with $\varphi_\lo=J(t)(V_\alpha*m)$. Suppose, for every $t\geq 1$, that
    \vspace{.2pt}
    \begin{subequations}\label{eq:assumption}
        \begin{gather}
        \|m(t)-m_{0}\|_{H^{N-1}} +\|m(t)-m_{0}\|_{W^{N-1,1}} \lesssim\varepsilon^2t^{-\alpha},\label{assumption:convergence}\\[.8em]
        \|m(t)\|_{H^N}+ \|m(t)\|_{W^{L,1}} +\|m_{0}\|_{H^N}+\|m_{0}\|_{W^{L,1}} \lesssim\varepsilon^2, \label{assumption:density}\\[.8em]
        \|\nabla_y\theta(t)\|_{L^\infty} +\|\nabla_y^2\theta(t)\|_{H^{N+1-\alpha}}\lesssim\varepsilon^2t^{1-\alpha}\label{assumption:action},\\[.8em]
        \|\mathcal J(t)\|_{W^{L,1}}+\|\mathcal J(t)\|_{W^{L,\infty}}\lesssim\varepsilon^2.\label{assumption:current}
    \end{gather}
    \end{subequations}
    Define
    \begin{equation}\label{eq:zero_action}
        \theta_0:=V_\alpha *m_0, \qquad \Gamma_0:=\frac{\nabla_y\theta_0}{1-\alpha}.
    \end{equation}
    For $1\leq n\leq L$, define recursively  
    \begin{equation}\label{eq:recurssion}
        \begin{split}
            m_n&:=-\frac{1}{n\alpha} \nabla_y\cdot \sum_{a+b=n-1}m_a\Gamma_b,\\
            \theta_n&:=V_\alpha *m_n +\frac{1}{2}\sum_{a+b=n-1} \Gamma_a\cdot \Gamma_b.
        \end{split}
    \end{equation}
 where for $1\leq n\leq L-1$,   
\begin{equation}\label{eq:Gamma_n}
        \Gamma_n:=\frac{\nabla_y \theta_n}{1-(n+1)\alpha},
    \end{equation}
     Define
    \begin{equation}\label{eq:Theta}
        \Theta(t)=\sum_{n=0}^L\Lambda_{(n+1)\alpha}(t)\,\theta_n.
    \end{equation}
    If $(L+1)\alpha<1$, then
    \begin{equation}\label{eq:decay_expansion}
        \big\|\partial_t (\theta(t)-\Theta(t))\big\|_{W^{1,\infty}}\lesssim\varepsilon^2 t^{-1-\alpha}+\varepsilon^4 t^{-(L+2)\alpha}.
    \end{equation}
    If $(L+1)\alpha=1$, then
    \begin{equation}\label{eq:decay_expansion_log}
        \big\|\partial_t (\theta(t)-\Theta(t))\big\|_{W^{1,\infty}}\lesssim\varepsilon^2 t^{-1-\alpha}+\varepsilon^4 t^{-1-\alpha}(1+\log t).
    \end{equation}
    In particular,  $\theta-\Theta$ converges in $W^{1,\infty}$ as $t\rightarrow \infty$.
\end{proposition}

\begin{proof}
    Sobolev embedding, the hypotheses \eqref{eq:assumption}, and the complementary-potential estimates of Lemma \ref{lem:est_varphi} imply
    \begin{align}
        \|m(t)\|_{W^{L,1}\cap W^{L,\infty} } +\|m_{0}\|_{W^{L,1}\cap W^{L,\infty}} &\lesssim\varepsilon^2,\\
        \|\Gamma(t)\|_{W^{L+1,\infty}} &\lesssim\varepsilon^2t^{1-\alpha}\label{eq:induction_hypo_Gamma}\\
        \|\varphi_\hi(t)\|_{W^{L+1,\infty}}\lesssim \|\varphi_\hi(t)\|_{H^N}&\lesssim \varepsilon^2 t^{-1}.\label{eq:induction_high_est}
    \end{align}
    The coefficient recursion consequently gives
    \begin{equation}\label{eq:recursive_bounds}
        \begin{split}
            \|m_n\|_{W^{L-n,1}\cap W^{L-n,\infty}}\,+\,\|\theta_n\|_{W^{L+2-n,\infty}}&\lesssim \varepsilon^{2n+2}, \quad 0\leq n\leq L,\\
        \big\|\Gamma_n\big\|_{W^{L+1-n,\infty}}&\lesssim \varepsilon^{2n+2}, \quad 0\leq n<L.
        \end{split}
    \end{equation}
    We prove these bounds by induction on $n$. The base case $n=0$ follows from the bound on $m_0$ in \eqref{assumption:density}, \eqref{eq:V_convolution}, \eqref{eq:nabla_V_convolution}, and the definition of $\Gamma_0$. 

    Suppose the bounds hold for all indices below $n$. In the formula for $m_n$ in \eqref{eq:recurssion}, the indices satisfy $a+b=n-1$. Hence, $m_a$ has at least $L+1-n$ derivatives in $L^1\cap L^\infty$, while $\Gamma_b$ has at least $L+2-n$ bounded derivatives. The product rule therefore controls $m_a\Gamma_b$ in $W^{L+1-n,1}\cap W^{L+1-n,\infty}$. Taking its divergence gives the required regularity of $m_n$. 

    The convolution estimate \eqref{eq:V_convolution} then places $V_\alpha*m_n$ in $W^{L+2-n,\infty}$. The quadratic terms $\Gamma_a\cdot\Gamma_b$ have the same regularity, since both factors have at least $L+2-n$ bounded derivatives. This proves the bound for $\theta_n$, and its defining relation gives the bound for $\Gamma_n$ when $n<L$. Finally, every product appearing at this step has size
    \[
       \varepsilon^{2a+2}\varepsilon^{2b+2} =\varepsilon^{2n+2},\qquad \text{for }\,\,a+b=n-1,
    \]
    which proves \eqref{eq:recursive_bounds}.

    \medskip

    \noindent \paragraph{\bf Induction setup}
    Define the remainders
    \begin{equation}\label{eq:remainder_density}
        r_{n}(t)=m(t)-\sum_{a=0}^nt^{-a\alpha}m_{a}, \quad 0\le n\le L,
    \end{equation}
    \begin{equation}\label{eq:remainder_action}
        q_{n}(t)=\Gamma(t)-\sum_{b=0}^nt^{1-(b+1)\alpha}\Gamma_{b}, \quad 0\le n<L.
    \end{equation}
    We prove, for $0\leq n<L$, that 
    \begin{subequations} \label{eq:induction}
    \begin{align}
        \|r_{n}(t)\|_{W^{L-1-n,1}\cap W^{L-1-n,\infty}}&\lesssim\varepsilon^2t^{-1}+\varepsilon^4t^{-(n+1)\alpha}, \label{eq:induction_density}\\[.6em]
        \|q_{n}(t)\|_{W^{L-n,\infty}} &\lesssim\varepsilon^2+\varepsilon^4\Lambda_{(n+2)\alpha}(t). \label{eq:induction_gradient}
    \end{align}
    \end{subequations}
    At each order we estimate the density by integrating from its limiting value, and then estimate the action by integrating from $t=1$. The normalization of the time factors gives
    \begin{equation}\label{remainder:constant}
         q_{n} =\nabla\left(\theta-\sum_{a=0}^n \Lambda_{(a+1)\alpha}\theta_{a}\right) -\sum_{a=0}^n \Gamma_{a}.
    \end{equation}
    The last sum is independent of time and is $\mathcal{O}(\varepsilon^2)$ in the required norm. 

    \medskip

    \noindent \paragraph{\bf Base step.} For $n=0$, the density equation in \eqref{eq:action--density} gives
    \begin{equation}\label{eq:density_base_identity}
         r_{0}(t)=\int_t^\infty s^{-2}\,\nabla_y\cdot\mathcal J(s)\,ds -\int_t^\infty s^{-2}\,\nabla_y\cdot (m\Gamma)(s)\,ds.
    \end{equation}
    For $p\in\{1,\infty\}$, the product rule and \eqref{eq:assumption} give
    \begin{align*}
        \|r_0(t)\|_{W^{L-1,p}} &\lesssim \int_t^\infty s^{-2}\left(\|\mathcal J(s)\|_{W^{L,p}}+\|m(s)\|_{W^{L,p}}\|\Gamma(s)\|_{W^{L,\infty}}\right)\,ds\\
        &\lesssim \varepsilon^2\int_t^\infty s^{-2}\,ds +\varepsilon^4\int_t^\infty s^{-1-\alpha}\,ds\,\,\,=\,\,\, \varepsilon^2t^{-1} +\frac{\varepsilon^4}{\alpha}t^{-\alpha}\lesssim \varepsilon^2 t^{-1}+\varepsilon^4t^{-\alpha}.
    \end{align*}
    Substituting this bound into the action equation and using the convolution estimate and the bound on $\varphi_\hi$, we obtain
    \begin{equation} \label{eq:density_base_check}
        \|\partial_t(\theta-\Lambda_\alpha\theta_0)(t) \|_{W^{L+1,\infty}} \lesssim\varepsilon^2t^{-1-\alpha}+\varepsilon^4t^{-2\alpha}.
    \end{equation}
    Integrating from 1 to $t$ proves the asserted bound for $q_0$, since $q_0(1)=-\Gamma_0$. 

    \medskip 

    \noindent \paragraph{\bf Inductive step.} Let $1\leq n\leq L$, and assume the remainder bounds in \eqref{eq:induction} at order $n-1$. The errors in the density-flux and the quadratic action term have the exact decompositions:
    \begin{equation}\label{eq:decomposition_density_errir}
        \begin{split}
            \mathcal{F}_n:&=m\Gamma-\sum_{a+b\leq n-1} t^{1-(a+b+1)\alpha}m_a\Gamma_b\\
            &= r_{n-1} \left(\sgm \right) +mq_{n-1} +\sum_{\substack{0\leq a,b<n\\a+b\geq n} } t^{1-(a+b+1)\alpha}m_a\Gamma_b,
        \end{split}
    \end{equation}
    \begin{equation}\label{eq:decomposition_action_errir}
        \begin{split}
            \mathcal{Q}_n:&= t^{-2}|\Gamma|^2 -\sum_{a+b\leq n-1} t^{-(a+b+2)\alpha}\Gamma_a\cdot \Gamma_b\\
            &=t^{-2}\left(\Gamma+\sgm\right)\cdot q_{n-1}+\sum_{\substack{0\leq a,b<n\\a+b\geq n}} t^{-(a+b+2)\alpha}\,\Gamma_a \cdot \Gamma_b.
        \end{split}
    \end{equation}
    All coefficient indices are nonnegative and empty sums are taken to vanish. The recursion cancels the retained terms, leaving
    \begin{equation}\label{eq:PDE_r}
        \partial_t r_n=t^{-2}\nabla_y\cdot \left(\mathcal{F}_n-\mathcal J\right),
    \end{equation}
    and 
    \begin{equation}\label{eq:PDE_action_difference}
        \partial_t\left(\theta-\sum_{a=0}^n \Lambda_{(a+1)\alpha}(t)\, \theta_a \right)=t^{-\alpha}V_\alpha *r_n-t^{-\alpha} \varphi_\hi +\frac{1}{2}\mathcal{Q}_n.
    \end{equation}
    Set $\beta=(n+1)\alpha\leq 1$. The coefficient bounds \eqref{eq:recursive_bounds} give
    \begin{equation}
        \left\|\,\sgm\,\right\|_{W^{L+1-n,\infty}}\lesssim\, \sum_{b=0}^{n-1}\varepsilon^{2b+2}t^{1-(b+1)\alpha} \lesssim\,\varepsilon^2t^{1-\alpha}.
    \end{equation}
    Together with \eqref{eq:induction_hypo_Gamma}, this implies
    \begin{equation}
        \left\|\,\Gamma(t)+ \sgm(t)\right\|_{W^{L+1-n,\infty}} \lesssim\varepsilon^2t^{1-\alpha}.
    \end{equation}
    We will also use
    \begin{equation}
        t^{1-\beta} =1+(1-\beta)\Lambda_\beta(t) \leq 1+\Lambda_\beta(t), \qquad \text{for }\,0<\beta\leq1.
    \end{equation}
For $p\in\{1,\infty\}$, applying the product rule to each term in \eqref{eq:decomposition_density_errir} gives
\begin{equation}\label{eq:est_F_n}
\begin{split}
    \|\mathcal F_n(t)\|_{W^{L-n,p}}&\lesssim \|r_{n-1}(t)\|_{W^{L-n,p}} \left\|\sgm(t)\right\|_{W^{L-n,\infty}}\\
    &\qquad + \|m(t)\|_{W^{L-n,p}} \|q_{n-1}(t)\|_{W^{L-n,\infty}}\\[.6em]
    &\qquad+ \sum_{\substack{0\le a,b<n\\a+b\ge n}} t^{1-(a+b+1)\alpha}
    \|m_a\|_{W^{L-n,p}} \|\Gamma_b\|_{W^{L-n,\infty}}\\
    &\lesssim \big(\varepsilon^2t^{-1}+\varepsilon^4t^{-n\alpha}\big) \varepsilon^2t^{1-\alpha}+\varepsilon^2\big(\varepsilon^2+\varepsilon^4\Lambda_\beta(t)\big)
    +\varepsilon^{2n+4}t^{1-\beta}\\
    &\lesssim\varepsilon^4 \big(1+\Lambda_\beta(t)\big).
\end{split}
\end{equation}
Here the finite sum is controlled using the coefficient bounds and
\[
\varepsilon^{2a+2}\varepsilon^{2b+2}\leq\varepsilon^{2n+4},  \qquad   t^{1-(a+b+1)\alpha}\leq t^{1-\beta}, \qquad \text{for }\,a+b\geq n.
\]
For the quadratic remainder, its factorization \eqref{eq:decomposition_action_errir} gives 
\begin{equation}\label{eq:est_Q_n}
\begin{split}
    \|\mathcal Q_n(t)\|_{W^{L+1-n,\infty}}
    &\lesssim t^{-2}\left\|\Gamma(t)+\sgm(t)\right\|_{W^{L+1-n,\infty}} \|q_{n-1}(t)\|_{W^{L+1-n,\infty}}\\
    &\qquad + \sum_{\substack{0\le a,b<n\\a+b\ge n}} t^{-(a+b+2)\alpha}\|\Gamma_a\|_{W^{L+1-n,\infty}} \|\Gamma_b\|_{W^{L+1-n,\infty}}\\
    &\lesssim \varepsilon^2t^{-1-\alpha} \big(\varepsilon^2+\varepsilon^4\Lambda_\beta(t)\big) +\varepsilon^{2n+4}t^{-(n+2)\alpha}\\
    &= t^{-1-\alpha} \left[ \varepsilon^4+\varepsilon^6\Lambda_\beta(t) +\varepsilon^{2n+4}t^{1-\beta} \right]\,\lesssim\, \varepsilon^4t^{-1-\alpha}\big(1+\Lambda_\beta(t)\big).
\end{split}
\end{equation}
Thus the flux estimate requires $L-n$ derivatives of $r_{n-1}$, while the quadratic estimate requires $L+1-n$ derivatives of $q_{n-1}$, exactly as supplied by the induction hypothesis \eqref{eq:induction}.

\medskip

\noindent \textit{\underline {If $n<L$,}} then $\beta<1$ and $1+\Lambda_\beta(t)\lesssim t^{1-\beta}$. The density remainder equation \eqref{eq:PDE_r} therefore satisfies
\begin{equation}
    \big\|\partial_t r_n(t)\big\|_{W^{L-1-n,1}\cap W^{L-1-n,\infty}}\lesssim\varepsilon^2 t^{-2}+\varepsilon^4 t^{-1-(n+1)\alpha},
\end{equation}
using \eqref{eq:est_F_n} and \eqref{assumption:current}. Integration from $t$ to infinity proves the
density estimate \eqref{eq:induction_density} at order $n$. 

Applying the convolution estimate \eqref{eq:V_convolution} with $s=L-1-n\ge0$, we bound
\begin{align*}
    \|V_\alpha*r_n(t)\|_{W^{L+1-n,\infty}} &\lesssim \|r_n(t)\|_{W^{L-1-n,1}} +\|r_n(t)\|_{W^{L-1-n,\infty}}\,\,\lesssim \,\,\varepsilon^2t^{-1} +\varepsilon^4t^{-(n+1)\alpha}.
\end{align*}
The gain of two derivatives provides exactly the spatial regularity needed in the action equation. Using \eqref{eq:PDE_action_difference}, \eqref{eq:induction_high_est} and \eqref{eq:est_Q_n}, we bound
\begin{align*}
    &\left\| \partial_t\left( \theta-\sum_{a=0}^{n}\Lambda_{(a+1)\alpha}\theta_a \right)(t) \right\|_{W^{L+1-n,\infty}}\\
    &\leq t^{-\alpha}\|V_\alpha*r_n(t)\|_{W^{L+1-n,\infty}} +t^{-\alpha}\|\varphi_{\mathrm{high}}(t)\|_{W^{L+1-n,\infty}} +\frac12\|\mathcal Q_n(t)\|_{W^{L+1-n,\infty}}\\
    &\lesssim \varepsilon^2t^{-1-\alpha} +\varepsilon^4t^{-(n+2)\alpha} +\varepsilon^4t^{-1-\alpha}\big(1+\Lambda_\beta(t)\big)\\
    &\lesssim \varepsilon^2t^{-1-\alpha}+\varepsilon^4t^{-(n+2)\alpha}.
\end{align*}
In the last step, we used $\beta=(n+1)\alpha<1$, since $n<L$, and hence
\[
t^{-1-\alpha}\big(1+\Lambda_\beta(t)\big)\lesssim t^{-1-\alpha}t^{1-\beta} =t^{-(n+2)\alpha}.
\]
Integrating from 1 to $t$ and using identity \eqref{remainder:constant} proves the action estimate. This closes the
induction for every $0\leq n<L$, including the case $(n+2)\alpha=1$. 

\medskip

\noindent \textit{\underline {If $n=L$,}} the induction controls $r_{L-1}$ only in $L^1\cap L^\infty$. We therefore \emph{convolve} the density equation \emph{before} estimating it to gain derivatives:
\begin{equation}
    \partial_t \left(V_\alpha*r_L \right)=t^{-2} \nabla_yV_\alpha * \big(\mathcal F_L-\mathcal J\big).
\end{equation}
The divergence now falls on the kernel, so no additional derivative of $r_{L-1}$ is required. With $\beta=(L+1)\alpha$, the convolution estimate \eqref{eq:nabla_V_convolution} and \eqref{eq:est_F_n} give\footnote{The boundary value $V_\alpha*r_L(t)\to 0$ as $t\rightarrow\infty$ follows from \eqref{assumption:convergence}, \eqref{eq:V_convolution} and \eqref{eq:recurssion}.}
\begin{equation}
    \big\|\partial_t \left(V_\alpha*r_L \right)(t)\big\|_{W^{1,\infty}}\lesssim \varepsilon^2 t^{-2}+\varepsilon^4 t^{-2}\big(1+\Lambda_\beta(t)\big).
\end{equation}
Since 
\begin{align*}
    \int_t^\infty s^{-2} \big(1+\Lambda_\beta(s)\big)\, ds=\frac{1+\Lambda_\beta(t)}{t}+\frac{t^{-\beta}}{\beta} \,\lesssim\, \frac{1+\Lambda_\beta(t)}{t},
\end{align*}
we obtain
\begin{equation}
    \big\|V_\alpha*r_L(t)\big\|_{W^{1,\infty}}\lesssim \varepsilon^2 t^{-1}+\varepsilon^4 t^{-1}\big(1+\Lambda_\beta(t)\big).
\end{equation}
Using \eqref{eq:PDE_action_difference}, \eqref{eq:induction_high_est} and \eqref{eq:est_Q_n}, we conclude that
\begin{align*}
    \|\partial_t(\theta-\Theta )(t)\|_{W^{1,\infty}} &\leq t^{-\alpha}\|V_\alpha*r_L(t)\|_{W^{1,\infty}}+t^{-\alpha}\|\varphi_{\mathrm{high}}(t)\|_{W^{1,\infty}} +\frac12\|\mathcal Q_L(t)\|_{W^{1,\infty}}\\
    &\lesssim \varepsilon^2t^{-1-\alpha} +\varepsilon^4t^{-1-\alpha}\bigl(1+\Lambda_\beta(t)\bigr).
\end{align*}
If $\beta<1$, this proves \eqref{eq:decay_expansion}; if $\beta=1$, then $\Lambda_\beta(t)=\log t$ and this proves \eqref{eq:decay_expansion_log}. Both bounds are integrable since $(L+2)\alpha>1$. As $\Theta(1)=\theta(1)=0$, there exists a real-valued function $\theta_{\mathrm{final}}\in W^{1,\infty}$ with 
\begin{equation}\label{eq:size_tail_action}
    \big\|\theta_{\mathrm{final}}\big\|_{W^{1,\infty}}\lesssim \varepsilon^2,
\end{equation}
such that
\begin{equation}\label{eq:convergence_action_final_state}
    \big\|\theta(t)-\Theta(t)-\theta_{\mathrm{final}}\big\|_{W^{1,\infty}}\lesssim
    \begin{cases}
        \varepsilon^2 t^{-\alpha} +\,\varepsilon^4 t^{1-(L+2)\alpha}, \qquad &(L+1)\alpha<1,\\[.4em]
         \varepsilon^2 t^{-\alpha} +\,\varepsilon^4 t^{-\alpha} \big(1+\log t\big), \qquad &(L+1)\alpha=1.
    \end{cases}
\end{equation}
This proves the proposition.
\end{proof}

\begin{remark}
   From the proof, we also have the regularity of the functions $\theta_\ell$: 
   \begin{equation*}
       \begin{split}
          \theta_0\in W^{L+2,\infty},\quad  \theta_1\in W^{L+1,\infty},\dots \theta_{L}\in W^{2,\infty}.
       \end{split}
   \end{equation*}
   As a consequence,  $\Theta(t,\cdot)\in W^{2,\infty}$ for fixed $t\geq 1$.
\end{remark}


\subsection{Modified scattering for Hartree}\label{subsec:hartree-scatter}

The corrected profile has already been shown to converge in Subsection \ref{subsec:boot_hartree}. We now replace the full action by its expansion and absorb the finite action remainder into the limiting profile.

\begin{proof}[Proof of Theorem \ref{thm:Hartree}]
Under the initial-data assumption \eqref{eq:initial-Hartree}, Lemma \ref{lem:close_hartree} and Proposition \ref{prop:hartree_h_convergence} give \eqref{assumption:convergence} and \eqref{assumption:action}, as well as
the $H^N$ bounds in \eqref{assumption:density}, with 
\[
m_0=m_\infty=|h_\infty|^2.
\]
The remaining density bounds follow from the product rule and Cauchy–Schwarz:
\begin{align*}
    \|m(t)\|_{W^{L,1}}+\|m_0\|_{W^{L,1}}\lesssim \|h(t)\|_{H^L}^2+\|h_\infty\|_{H^L}^2\lesssim\varepsilon_0^2. 
\end{align*}
For the current $\mathcal J=\Im (\bar h\nabla_y h)$, \eqref{eq:current_est} and $N\geq L+3$ give
\[
\big\|\mathcal J(t)\big\|_{W^{L,1}\cap W^{L,\infty}}\lesssim\big\|\mathcal{J}(t)\big\|_{W^{N-1,1}\cap H^{N-1}}\lesssim\varepsilon_0^2.
\]
Thus all four assumptions in \eqref{eq:assumption} hold, and Proposition \ref{prop:expansion} applies with its smallness parameter a fixed multiple of $\varepsilon_0$. Let $\Theta$ be the action defined by \eqref{eq:Theta}, with coefficients generated from $m_0=m_\infty$ by \eqref{eq:recurssion}. By \eqref{eq:convergence_action_final_state}, there is a real-valued $\theta_{\mathrm{final}}\in W^{1,\infty}$ such that 
\begin{equation}
    \big\|\theta(t)-\Theta(t)-\theta_{\mathrm{final}}\big\|_{W^{1,\infty}}\lesssim\varepsilon_0^2 \mathfrak r_\alpha(t),
\end{equation}
where $\mathfrak{r}_{\alpha}$ is defined in \eqref{eq:errorfunction}. 
Define the scattering profile by $\widetilde{h}_\infty:=e^{-i\theta_{\mathrm{final}}}h_\infty$. Since multiplication by this phase preserves the modulus, the limiting density is unchanged:
\[
\big|\widetilde{h}_\infty\big|^2=|h_\infty|^2=m_\infty.
\]
Consequently, the coefficients of $\Theta$ are also the coefficients generated by the density of $\widetilde h_\infty.$ The available phase regularity also gives $\widetilde h_\infty\in H^1\cap L^\infty$, with
\[
\big\|\widetilde h_\infty\big\|_{H^1\cap L^\infty}\lesssim \big(1+\|\nabla_y \theta_\mathrm{final}\|_{L^\infty}\big) \|h_\infty\|_{H^N}\lesssim \varepsilon_0.
\]
The reconstruction formula reads 
\[
u(t,x)=t^{-3/2}e^{i\frac{|x|^2}{2t}}e^{-i\theta\left(t,\frac{x}{t}\right)}h\left(t,\frac{x}{t}\right).
\]
For $2\leq p\leq \infty$, Proposition \ref{prop:hartree_h_convergence} and Sobolev embedding using $H^{N-1}$ give $\|h(t)-h_\infty\|_{L^p}\lesssim\varepsilon_0 t^{-\alpha}$. Using $|e^{ia}-e^{ib}|\leq |a-b|$ for real numbers $a,b$ and $t^{-\alpha}\leq \mathfrak r_\alpha(t)$, we obtain
\begin{align*}
    \big\|e^{-i(\theta(t)-\Theta(t))}h(t)-\widetilde h_\infty\big\|_{L^p}&\leq \|h(t)-h_\infty\|_{L^p}+\|\theta(t)-\Theta(t)-\theta_{\mathrm{final}}\|_{L^\infty}\|h_\infty\|_{L^p}\\
    &\lesssim\varepsilon_0 t^{-\alpha}+\varepsilon_0^3 \mathfrak r_\alpha(t)\,\lesssim\, \varepsilon_0 \mathfrak r_\alpha(t).
\end{align*}
The change of variables $x=ty$ therefore yields
\begin{equation}
    \left\|u(t,x)-t^{-3/2}e^{i\frac{|x|^2}{2t}}e^{-i\Theta\left(t,\frac{x}{t}\right)}\widetilde h_\infty\left(\frac{x}{t}\right) \right\|_{L^p_x} \lesssim\varepsilon_0 t^{-\frac{3}{2}+\frac{3}{p}} \mathfrak r_\alpha(t).
\end{equation}
In particular taking $p=\infty$ and renaming $\widetilde{h}_{\infty}$ as $h_{\infty}$, this gives the stated uniform asymptotics in Theorem \ref{thm:Hartree}.
\end{proof}


\subsection{Modified scattering for Vlasov--Riesz}\label{subsec:VR-scatter} As in the Hartree argument in Subsection \ref{subsec:hartree-scatter}, we combine the convergence of the corrected profile $h$, established in Subsection \ref{subsec:boot_VR}, with the action expansion in Proposition \ref{prop:expansion} to obtain the final asymptotics.

\begin{proof}[Proof of Theorem \ref{thm:VR}]
Under the initial-data assumption \eqref{eq:initial-VR}, Lemma \ref{lem:propagation} and Proposition \ref{prop:convergence_h2}, together with Lemma \ref{lem:action_est}, give \eqref{assumption:convergence}-\eqref{assumption:current}, with
\[
m_0=m_\infty=\int h_\infty^2\, d\eta.
\]
Proposition \ref{prop:expansion} and \eqref{eq:convergence_action_final_state} provide the action $\Theta$ generated by $m_0$ and a real-valued $\theta_{\mathrm{final}}\in W^{1,\infty}$ such that   
\begin{equation}
    \big\|\theta(t)-\Theta(t)-\theta_{\mathrm{final}}\big\|_{W^{1,\infty}}\lesssim\varepsilon_0^2 \mathfrak r_\alpha(t).
\end{equation}

Define the scattering profile by $\widetilde h_\infty(y,\eta):=h_\infty(y,\eta+\nabla_y\theta_{\mathrm{final}}(y))$. Since translation in $\eta$ preserves the integral, the limiting density is unchanged:
\begin{equation*}
\begin{split}
\int_{\mathbb R^3}\widetilde h_\infty(y,\eta)^2\,d\eta =\int_{\R^3}h_\infty(y,\eta)^2\,d\eta =m_\infty(y).
\end{split}
\end{equation*}
This, as in the Hartree case, again leaves every coefficient of $\Theta$ unchanged. Fubini's theorem also shows that a $y$-dependent translation in $\eta$ preserves each $L^p_{y,\eta}$ norm. In particular, $\widetilde h_\infty \in L^2\cap L^\infty$ and 
\[
\big\|\widetilde h_\infty\big\|_{L^p}=\big\| h_\infty\big\|_{L^p}\lesssim\varepsilon_0, \quad \text{for }\, 2\leq p\leq \infty.
\]
The definitions of $g$ and $h$ give the reconstruction formula
\begin{equation*}
\begin{split}
     f(t,x,v) = h\left(t,\frac{x}{t},tv-x+\Gamma \left(t,\frac{x}{t}\right) \right).
\end{split}
\end{equation*}
By Proposition \ref{prop:convergence_h2}, \eqref{eq:kinetic_propagation}, and interpolation between $L^2$ and $L^\infty$, 
\begin{align*}
    \|h(t)-h_\infty\|_{L^p}\lesssim\varepsilon_0 t^{-\alpha}, \quad \|\nabla_\eta h(t)\|_{L^p}\lesssim\varepsilon_0,\quad \text{for }\,2\leq p\leq\infty. 
\end{align*}
Using the fundamental theorem of calculus in $\eta$ on $h(t)$ and the translation invariance just noted, we obtain
\begin{align*}
    &\left\|h\left(t,y,\eta+\Gamma(t,y)-\nabla_y\Theta(t,y)\right)-\widetilde h_\infty\left(y,\eta\right)\right\|_{L^p_{y,\eta}}\\
    &\quad \leq\|h(t)-h_\infty\|_{L^p} +\|\Gamma(t)-\nabla_y\Theta(t)-\nabla_y \theta_{\mathrm{final}}\|_{L^\infty}\|\nabla_\eta h(t)\|_{L^p}\\
    &\quad\lesssim\varepsilon_0t^{-\alpha}+\varepsilon_0^3\mathfrak r_\alpha(t)\lesssim\varepsilon_0\mathfrak r_\alpha(t).
\end{align*}
Finally, the change of variables $(x,v)=(ty,y+t^{-1}(\eta-\nabla_y\Theta(t,y)))$ has Jacobian determinant one. Hence, 
\begin{equation*}
\begin{split}
\left\|f(t,x,v)-\widetilde h_\infty\left(\frac{x}{t},tv-x+\nabla_y\Theta\left(t,\frac{x}{t}\right)\right)\right\|_{L^p_{x,v}}\lesssim\varepsilon_0\mathfrak r_\alpha(t),\qquad 2\le p\le\infty,
\end{split}
\end{equation*}
which proves the stated\footnote{To match the residual-momentum convention in Theorem \ref{thm:VR}, one sets $f_{\infty}(x,v)=\widetilde{h}_{\infty}(x,-v)$} asymptotics in Theorem \ref{thm:VR}.

\end{proof}

\section{The Vlasov--Hartree system}\label{sec:VH}
In this section, we apply the arguments of Sections \ref{sec:boot} and \ref{sec:asymptotics} to the coupled system \eqref{VH}. We retain the component labels introduced in Section \ref{sec:idea}. Thus the interaction potentials are $\varphi_\hs=V_\alpha*m_\ks$ and $\varphi_\ks=V_\alpha*m_\hs$, while each density is transported by its own action. The profile equations remain \eqref{eq:PDE_h_H} and \eqref{eq:PDE_h_K}, with these exchanged sources. We first establish the global profile bounds and convergence, and then obtain the asymptotic actions by the coupled version of the recursion \eqref{eq:recurssion} in Proposition \ref{prop:expansion}. 

Following the discussion of Section \ref{sec:idea}, using the profiles \eqref{eq:g}, the corrected profiles \eqref{eq:h}, and the identities \eqref{eq:theta} and \eqref{eq:continuity-m}, we have the \emph{coupled action--density system}
\begin{subequations}\label{eq:VH_action_density}
\begin{equation}\label{eq:VH_action_density_H}
\left\{\begin{aligned}
\partial_t\theta_\hs&=\frac{|\nabla_y\theta_\hs|^2}{2t^2}+t^{-\alpha}J(t)(V_\alpha*m_\ks),\\
\partial_t m_\hs&=\frac1{t^2}\nabla_y\cdot(m_\hs\nabla_y\theta_\hs)-\frac1{t^2}\nabla_y\cdot\mathcal J_\hs.
\end{aligned}\right.
\end{equation}
\begin{equation}\label{eq:VH_action_density_K}
\left\{\begin{aligned}
\partial_t\theta_\ks&=\frac{|\nabla_y\theta_\ks|^2}{2t^2}+t^{-\alpha}J(t)(V_\alpha*m_\hs),\\
\partial_t m_\ks&=\frac1{t^2}\nabla_y\cdot(m_\ks\nabla_y\theta_\ks)-\frac1{t^2}\nabla_y\cdot\mathcal J_\ks.
\end{aligned}\right.
\end{equation}
\end{subequations}
Here $\hs$ and $\ks$ label the wave and kinetic components, respectively. For each $k\in\{\hs,\ks\}$, $\theta_k$ is the action correcting the profile $h_k$, while $m_k$ and $\mathcal J_k$ are its density and residual current. We normalize $\theta_k(1)=0$ and use the regularizer \eqref{eq:Jt} and the potential decomposition \eqref{eq:decomposition}. The coupling enters through the exchanged source densities in the action equations: $m_\ks$ drives $\theta_\hs$, and $m_\hs$ drives $\theta_\ks$. Each density equation, however, involves the action and current of its own component. We will also use the evolution equations for the corrected profiles $h_\hs$ and $h_\ks$. The wave profile satisfies
\begin{equation}\label{eq:VH_profile_H}
i\partial_t h_\hs+\frac{1}{2t^2}\Delta_y h_\hs-\frac{i}{t^2}\mathcal T_{\theta_\hs}h_\hs=t^{-\alpha}\varphi_{\hs,\hi}h_\hs,
\end{equation}
where $\mathcal T_{\theta_\hs}:=\Gamma_\hs\cdot\nabla_y+\frac12\nabla_y\cdot\Gamma_\hs$ and $\Gamma_k=\nabla_y\theta_k$. The kinetic profile satisfies the Hamiltonian transport equation
\begin{equation}\label{eq:VH_profile_K}
\partial_t h_\ks+\{h_\ks,\mathcal K_\ks\}=0,
\end{equation}
with
\begin{equation}\label{eq:VH_Hamiltonian}
\mathcal K_\ks(t,y,\eta):=t^{-2}\left(\frac{|\eta|^2}{2}-\Gamma_\ks(t,y)\cdot\eta\right)+t^{-\alpha}\varphi_{\ks,\hi}(t,y).
\end{equation}
Thus the corrected profile equations retain the same form as in the scalar problems, with the interaction potentials generated by the opposite component.

\subsection{Global bounds and limiting profiles} Let $N\geq L+3$, so that $N\geq 4$. Local well-posedness follows by a standard iteration of the Schr\"odinger and characteristic equations. We work on the maximal interval of existence; the coupled a priori estimates below yield global continuation.

For sufficiently small $\varepsilon_0$, the initial-data assumptions \eqref{eq:initial-Hartree_VH} and \eqref{eq:initial-VR} of Theorem \ref{thm:VH} give a solution to the coupled system \eqref{VH} on $[0,1]$. Using the Galilean fields as in \eqref{eq:Galilean_g} and the free-transport weight $\langle x-tv\rangle^{2N+10}$ as in \eqref{eq:weighted_small_time}, the standard weighted local estimates provide the bounds needed to initialize the profiles at $t=1$. Since $\theta_{\mathrm H}(1)=\theta_{\mathrm K}(1)=0$, we have $h_{\mathrm H}(1)=g_{\mathrm H}(1)$ and $h_{\mathrm K}(1)=g_{\mathrm K}(1)$. 
The identities
\begin{equation}
    \|\partial_y^\mu g_\hs (1)\|_{L^2_y}= \|G(1)^\mu u(1)\|_{L^2_x},\qquad g_\ks(1,y,\eta)=f(1,y,y+\eta),
\end{equation}
together with $\la x-v\ra=\la\eta\ra$ under the latter change of variables, therefore give
\begin{equation}\label{eq:small_time_VH_data}
    \begin{split}
        \|h_\hs(1)\|_{H^N_y}&+ \sum_{|\mu|+|\nu|\leq N-2}\big\| \langle\eta\rangle^{2N+10} \partial_y^\mu\partial_\eta^\nu h_\ks(1) \big\|_{L^\infty_{y,\eta}}\\
        &\qquad\qquad+ \sum_{|\mu|+|\nu|\leq N} \big\| \langle\eta\rangle^{2N+10} \partial_y^\mu\partial_\eta^\nu h_\ks(1) \big\|_{L^2_{y,\eta}} \lesssim \varepsilon_0.
    \end{split}
\end{equation}
These estimates supply the initial bounds for the coupled bootstrap below. 

As in Section \ref{sec:boot}, we impose the bootstrap \eqref{boot} for both components: for $k\in\{\hs,\ks\}$ and $1\leq t\leq T$,
\begin{equation}\label{eq:boot-VH}
    \begin{split}
        \|m_k(t)\|_{H^N}\leq \varepsilon^2, \qquad \|\nabla_y\Gamma_k(t)\|_{H^{N+1-\alpha}}\leq \varepsilon^2 t^{1-\alpha}.
    \end{split}
\end{equation}
Here $0<\varepsilon_0\leq c\varepsilon$ and $\varepsilon$ is sufficiently small. We close the bootstrap in the following lemma.
\begin{lemma}\label{lem:VH_bounds}
Under the assumptions on initial data in \eqref{eq:initial-Hartree_VH} and \eqref{eq:initial-VR}, the corrected profiles and actions exist for all $t\geq 1$ and the profiles satisfy
\begin{equation}\label{eq:VH_profile_bounds}
\|h_\hs(t)\|_{H^N}+\sum_{|\mu|+|\nu|\le N-2}\|\langle\eta\rangle^4\partial_y^\mu\partial_\eta^\nu h_\ks(t)\|_{L^\infty}+\sum_{|\mu|+|\nu|\le N}\|\langle\eta\rangle\partial_y^\mu\partial_\eta^\nu h_\ks(t)\|_{L^2}\lesssim\varepsilon_0.
\end{equation}
\vspace{.4em}
Moreover, for $k\in\{\hs,\ks\}$,
\begin{equation}\label{eq:VH_density_action_bounds}
\begin{split}
\|m_k(t)\|_{L^1}+\|m_k(t)\|_{H^N}&\lesssim\varepsilon_0^2,\\
\|\Gamma_k(t)\|_{L^\infty}+\|\nabla_y\Gamma_k(t)\|_{H^{N+1-\alpha}} &\lesssim\varepsilon_0^2t^{1-\alpha}.
\end{split}
\end{equation}
The currents satisfy
\begin{equation}\label{eq:VH_current_est}
    \begin{split}
        \|\mathcal{J}_k\|_{W^{N-1,1}\cap H^{N-1}}\les \varepsilon_0^2.
    \end{split}
\end{equation}
\end{lemma}
\begin{proof}
First, conservation of the two masses gives $\|m_k(t)\|_{L^1}\les \varepsilon_0^2$ for $k\in \{\hs,\ks\}$.    Since Lemmas \ref{lem:est_varphi} and \ref{lem:action_est} depend only on the norms of the source density, their proofs apply with the sources exchanged. In particular,
\begin{equation*}
\begin{split}
\|\varphi_{k,\hi}(t)\|_{H^N}&\lesssim\varepsilon^2t^{-1},\qquad \|\nabla_y\varphi_{k,\hi}(t)\|_{H^N}\lesssim\varepsilon^2t^{-1+\frac{\alpha}{2}},\\ \|\Gamma_k(t)\|_{L^\infty}+\|\nabla_y\Gamma_k(t)\|_{H^{N+1-\alpha}}&\lesssim\varepsilon^2t^{1-\alpha},\qquad \|\nabla_y\Gamma_k(t)\|_{H^{N+1}}\lesssim\varepsilon^2t^{1-\frac{\alpha}{2}}.
\end{split}
\end{equation*}
For the wave profile, the energy argument in Lemma \ref{lem:close_hartree} gives
\begin{equation*}
\begin{split}\frac{d}{dt}\|h_\hs(t)\|_{H^N}&\lesssim\left(t^{-2}\|\nabla_y\Gamma_\hs(t)\|_{H^N}+t^{-\alpha}\|\varphi_{\hs,\hi}(t)\|_{H^N}\right)\|h_\hs(t)\|_{H^N}\\&\lesssim\varepsilon^2t^{-1-\alpha}\|h_\hs(t)\|_{H^N}.
\end{split}\end{equation*}
Gr\"onwall's inequality therefore yields $\|h_\hs(t)\|_{H^N}\lesssim\varepsilon_0$. Hence $\|m_\hs(t)\|_{H^N}\les \varepsilon_0^2$.

For the kinetic profile $h_{\ks}$, we repeat the proof of Lemma \ref{lem:propagation}, retaining $N$ derivatives in $L^2$ and $N-2$ derivatives in the weighted $L^{\infty}$ estimates. This adjustment is needed because the source is now $m_\hs$: its $H^N$ bound does not directly give a $W^{N-1,\infty}$ bound, so \eqref{eq:phi-high-oo} is not available at the same derivative order. We therefore replace the scalar equation \eqref{eq:phi-high-oo} by the following consequence of Sobolev embedding and \eqref{eq:varphi_hi_derivative}:
\begin{equation*}
    \begin{split}
        \|\nabla_y\varphi_{\ks,\hi}(t)\|_{W^{N-2,\infty}}\les \|\nabla_y \varphi_{\ks,\hi}(t)\|_{H^N}\les \varepsilon^2 t^{-1+\frac{\alpha}{2}}.
    \end{split}
\end{equation*}
Write $h_\ks=\sum_{\ell\ge0}h_{\ks,\ell}$ using the transported localization from Lemma \ref{lem:propagation}. For the complementary potential, the additional terms in the differentiated $L^2$ estimates have $|\mu'|=N-1$. We estimate them together with those having $|\mu'|=N$. On the late interval $t\ge2^\ell$, the support property and Cauchy--Schwarz give
\begin{equation*}
\begin{split}
&t^{-\alpha}2^{\ell|\nu|}\left\|\left\{\partial_y^{\mu-\mu'}\partial_\eta^\nu h_{\ks,\ell},\partial_y^{\mu'}\varphi_{\ks,\hi}\right\}\right\|_{L^2}\\&\quad\lesssim t^{-\alpha}2^{3\ell/2-\ell}\|\nabla_y\partial_y^{\mu'}\varphi_{\ks,\hi}\|_{L^2_y}\sum_{i=1}^3 2^{\ell(|\nu|+1)}\|\partial_y^{\mu-\mu'}\partial_\eta^{\nu+e_i}h_{\ks,\ell}\|_{L^\infty}\\&\quad\lesssim\varepsilon^2t^{-1-\frac{\alpha}{2}}2^{3\ell/2-\ell}\sum_{i=1}^3 2^{\ell(|\nu|+1)}\|\partial_y^{\mu-\mu'}\partial_\eta^{\nu+e_i}h_{\ks,\ell}\|_{L^\infty}.
\end{split}
\end{equation*}
The derivative order on the last factor is at most $N-|\mu'|+1\le2\le N-2$. The top action terms are controlled as in \eqref{eq:late_top_action_source}, using the retained first derivatives in $L^\infty$; lower action derivatives remain bounded coefficients. On $1\leq t\leq 2^\ell$, the unregularized field estimates \eqref{eq:nabla_varphi} and \eqref{eq:nabla2_varphi} give the same early-interval argument, with $N-2$ in place of $N-1$ in the $L^\infty$ sum. The change between $g_{\ks,\ell}$ and $h_{\ks,\ell}$ is unchanged. Applying Lemma \ref{lem:hamiltonian_transport} and summing the differentiated estimates as in Lemma \ref{lem:propagation}, we obtain
\begin{equation*}
2^{3\ell/2}\sum_{|\mu|+|\nu|\le N-2}\|\partial_y^\mu\partial_\eta^\nu h_{\ks,\ell}(t)\|_{L^\infty} +\sum_{|\mu|+|\nu|\le N}\|\partial_y^\mu\partial_\eta^\nu h_{\ks,\ell}(t)\|_{L^2}\lesssim\varepsilon_0\,2^{-11\ell/2}.
\end{equation*}
Since $\langle\eta\rangle\lesssim2^\ell$ on the support of each differentiated piece, summing in $\ell$ proves \eqref{eq:VH_profile_bounds}, using $\sum_{\ell\ge0}(2^{-3\ell}+2^{-9\ell/2})<\infty$.

The product estimates used in Section \ref{sec:boot} now give $\|m_\hs(t)\|_{H^N}+\|m_\ks(t)\|_{H^N}\lesssim\varepsilon_0^2$. For the kinetic density, the factor with fewer derivatives has order at most $\lfloor N/2\rfloor\le N-2$, so the weighted $L^\infty$ bound above suffices. Applying Lemma \ref{lem:action_est} once more gives, for $k\ne j$,
\begin{equation*}
\|\Gamma_k(t)\|_{L^\infty}+\|\nabla_y\Gamma_k(t)\|_{H^{N+1-\alpha}}\lesssim\int_1^t s^{-\alpha}\left(\|m_j(s)\|_{L^1}+\|m_j(s)\|_{H^N}\right)\,ds\lesssim\varepsilon_0^2t^{1-\alpha}.
\end{equation*}
Choosing $\varepsilon_0\le c\varepsilon$ with $c$ sufficiently small improves both bootstrap assumptions. 

These estimates improve both bootstrap assumptions. The profile and action bounds control the kinetic regularity on every bounded time interval, while the field estimates and the differentiated Schrödinger equation propagate the $H^{N-1}_x$ regularity of $u$. The local continuation criterion therefore excludes a finite maximal lifespan and extends the solution and the auxiliary actions globally. Moreover, we have \eqref{eq:VH_profile_bounds} and \eqref{eq:VH_density_action_bounds} for all $t\geq 1$. 

Finally, the estimates \eqref{eq:VH_current_est} follow by the same arguments used to prove \eqref{eq:current_est} and \eqref{eq:kinetic_currents}.

\end{proof}

Using these estimates, we now prove the convergence of the profiles and the densities.
\begin{proposition}\label{prop:VH_convergence}
Under the assumptions on initial data in \eqref{eq:initial-Hartree_VH} and \eqref{eq:initial-VR}, there exist limiting profiles $h_{\hs,\infty}\in H^{N-1}_y$ and $h_{\ks,\infty}\in H^{N-1}_{y,\eta}\cap L^\infty_{y,\eta}$ such that
\begin{equation*}
\sum_{k\in\{\hs,\ks\}}\left(\|h_k(t)-h_{k,\infty}\|_{H^{N-1}}+\|h_k(t)-h_{k,\infty}\|_{L^\infty}\right)\lesssim\varepsilon_0t^{-\alpha}.
\end{equation*}
Define
\begin{equation}
    m_{\hs,\infty}:=|h_{\hs,\infty}|^2, \qquad m_{\ks,\infty}:=\int h_{\ks,\infty}^2\,d\eta.
\end{equation}
Then
\begin{equation*}
\|m_k(t)-m_{k,\infty}\|_{H^{N-1}\cap W^{N-1,1}}\lesssim\varepsilon_0^2t^{-\alpha}, \quad \|m_{k,\infty}\|_{H^N}\lesssim\varepsilon_0^2, \quad k\in \{\hs,\ks\}. 
\end{equation*}

\end{proposition}
\begin{proof}
    The convergence arguments of Propositions \ref{prop:hartree_h_convergence} and \ref{prop:convergence_h2} use only the profile bounds and the field and action estimates already established. In particular, the wave equation and the kinetic Hamiltonian equation give
\begin{equation*}
\begin{split}
\sup_{\tau\geq t}\|h_\hs(\tau)-h_\hs(t)\|_{H^{N-1}}&\les \varepsilon_0 t^{-1/2}+\varepsilon_0^3 t^{-\alpha} ,  \quad \tau\geq t\geq 1 \\ 
\|\partial_t h_\ks(t)\|_{H^{N-1}}+\|\partial_t h_\ks(t)\|_{L^\infty}&\lesssim\varepsilon_0t^{-2}+\varepsilon_0^3t^{-1-\alpha}.
\end{split}
\end{equation*}
Since $\alpha\leq \frac{1}{2}$, the first estimate gives convergence in $H^{N-1}$, and Sobolev embedding gives convergence in $L^\infty$ for the wave profile. Integrating the second estimate proves the kinetic convergence. The reduced $L^\infty$ derivative range in \eqref{eq:VH_profile_bounds} suffices, since only first derivatives are needed for the $L^\infty$ estimate. 

The density estimates follow from the same quadratic product estimates as in Propositions \ref{prop:hartree_h_convergence} and \ref{prop:convergence_h2}. Finally, weak lower semi-continuity gives $$\|m_{k,\infty}\|_{H^N}\le\liminf_{t\to\infty}\|m_k(t)\|_{H^N}\lesssim\varepsilon_0^2.$$
 This finishes the proof.
\end{proof}

\subsection{Coupled asymptotic expansions}
We now  apply the recursive argument of Subsection \ref{subsec:expansion} with exchanged sources. Lemma \ref{lem:VH_bounds} and Proposition \ref{prop:VH_convergence} verify the bounds \eqref{assumption:convergence}--\eqref{assumption:current} for both components. Indeed, the required $W^{L,1}$ density bounds follow from the profile bounds and Cauchy-Schwarz, while \eqref{eq:VH_current_est} and $N-1\geq L+2$ give the $W^{L,1}\cap W^{L,\infty}$ current bounds. 

The coefficient construction differs from the scalar case only in the source index. For $k,j\in\{\hs,\ks\}$ with $k\ne j$, set
\begin{equation*}
m_{k,0}:=m_{k,\infty},\qquad \theta_{k,0}:=V_\alpha*m_{j,0},\qquad \Gamma_{k,0}:=\frac{\nabla_y\theta_{k,0}}{1-\alpha}.
\end{equation*}
For $1\le n\le L$, first determine the two density coefficients and then the two action coefficients by
\begin{equation}\label{eq:VH_recursion}
\begin{split}
    m_{k,n}&:=-\frac1{n\alpha}\nabla_y\cdot\sum_{a+b=n-1}m_{k,a}\Gamma_{k,b},\\ \theta_{k,n}&:=V_\alpha*m_{j,n}+\frac12\sum_{a+b=n-1}\Gamma_{k,a}\cdot\Gamma_{k,b}.
\end{split}
\end{equation}
For $n<L$, define 
\[
\Gamma_{k,n}:=\frac{\nabla_y\theta_{k,n}}{1-(n+1)\alpha}.
\]
This recursion is explicit: the density coefficients at order $n$ involve only coefficients of lower order, and the action coefficients at order $n$ use the density coefficients just determined. Define $\Theta_k$ by \eqref{eq:action_asymptotics} with these coefficients.

We now repeat the proof of Proposition \ref{prop:expansion} simultaneously for the two components. The density remainder equation for component $k$ is unchanged. In the action remainder equation, the convolved density remainder is that of component $j\ne k$. Thus, at each order, we first prove both density estimates and then both action estimates. Taking the maximum over the two components gives the same induction bounds and the same time weights as in Subsection \ref{subsec:expansion}. In particular,
\begin{equation*}
\sum_{k\in\{\hs,\ks\}}\|\partial_t(\theta_k-\Theta_k)(t)\|_{W^{1,\infty}}\lesssim\varepsilon_0^2t^{-1-\alpha}+\varepsilon_0^4 t^{-1-\alpha}(1+\Lambda_{(L+1)\alpha}(t)).
\end{equation*}
Since $(L+2)\alpha>1$, there exist real-valued functions $\theta_{k,\mathrm{final}}\in W^{1,\infty}$ with $\|\theta_{k,\mathrm{final}}\|_{W^{1,\infty}}\lesssim \varepsilon_0^2$ such that
\begin{equation}\label{eq:VH_action_convergence}
\sum_{k\in\{\hs,\ks\}}\|\theta_k(t)-\Theta_k(t)-\theta_{k,\mathrm{final}}\|_{W^{1,\infty}}\lesssim\varepsilon_0^2\mathfrak r_\alpha(t).
\end{equation}

\medskip

We now combine the profile convergence in Proposition \ref{prop:VH_convergence} with \eqref{eq:VH_action_convergence}, as in Subsections \ref{subsec:hartree-scatter} and \ref{subsec:VR-scatter}. Define the scattering profiles by 
\[
h_\infty(y):=e^{-i\theta_{\hs,\mathrm{final}}(y)}h_{\hs,\infty}(y), \qquad f_\infty(y,\eta):=h_{\ks,\infty}(y,-\eta+\nabla_y\theta_{\ks,\mathrm{final}}(y)).
\]
These profiles belong to $H^1_y\cap L^\infty_y$ and $L^2_{y,\eta}\cap L^\infty_{y,\eta}$, respectively, with norms $\lesssim \varepsilon_0$. These transformations preserve the limiting densities: $|h_\infty|^2=m_{\hs,\infty}$ and $\int f_\infty^2\,d\eta=m_{\ks,\infty}$. They therefore leave all coefficients of the two asymptotic actions unchanged. The reconstruction formulas are
\begin{equation*}u(t,x)=t^{-3/2}e^{i|x|^2/(2t)}e^{-i\theta_\hs(t,x/t)}h_\hs(t,x/t),\qquad f(t,x,v)=h_\ks\left(t,\frac xt,tv-x+\Gamma_\ks\left(t,\frac xt\right)\right).\end{equation*}
For $2\le p\le\infty$, the same argument as in Subsection \ref{subsec:hartree-scatter} gives
\begin{equation*}
\begin{split}
&t^{3/2-3/p}\left\|u(t,x)-t^{-3/2}e^{i|x|^2/(2t)}e^{-i\Theta_\hs(t,x/t)}h_\infty(x/t)\right\|_{L^p_x}\\
&\quad\lesssim\|h_\hs(t)-h_{\hs,\infty}\|_{L^p}+\|\theta_\hs(t)-\Theta_\hs(t)-\theta_{\hs,\mathrm{final}}\|_{L^\infty}\|h_{\hs,\infty}\|_{L^p}\\[.1em]
&\quad \lesssim\varepsilon_0\mathfrak r_\alpha(t).
\end{split}
\end{equation*}
For the kinetic component, the change of variables $(x,v)\mapsto(x/t,tv-x)$ preserves phase-space volume. Using the fundamental theorem of calculus in $\eta$, as in Subsection \ref{subsec:VR-scatter}, we obtain
\begin{equation*}
\begin{split}
&\left\|f(t,x,v)-f_\infty\left(\frac xt,x-tv-\nabla_y\Theta_\ks\left(t,\frac xt\right)\right)\right\|_{L^p_{x,v}}\\
&\quad\lesssim\|h_\ks(t)-h_{\ks,\infty}\|_{L^p}+\|\nabla_y\theta_\ks(t)-\nabla_y\Theta_\ks(t)-\nabla_y\theta_{\ks,\mathrm{final}}\|_{L^\infty}\|\nabla_\eta h_\ks(t)\|_{L^p}\\
&\quad \lesssim\varepsilon_0\mathfrak r_\alpha(t).
\end{split}
\end{equation*}
Hence, we finish the proof Theorem \ref{thm:VH}.


\end{document}